\documentclass[12pt,reqno]{amsart}

\allowdisplaybreaks

\usepackage{amssymb,amsmath,amsthm,amscd,ifthen,xr,array}

\usepackage{graphicx}
\usepackage{xcolor}
\usepackage{scalerel}
\usepackage{tikz}
\usetikzlibrary{decorations.shapes}
\usetikzlibrary{decorations.pathreplacing}
\usetikzlibrary{matrix}
\usepackage[pagewise,mathlines,displaymath]{lineno}

\def\ch{\mathop{\hbox{\rm ch}}\nolimits}

\def\g{\mathfrak g}
\def\z{\mathfrak z}

\def\h{\mathfrak h}
\def\sp{\mathfrak {sp}}

\def\fraks{\mathfrak s}
\def\u{\mathfrak u}

\def\p{\mathfrak p}

\def\R{\mathbb{R}}
\def\C{\mathbb{C}}
\def\ZZ{\mathbb{Z}}
\def\HH{\mathbb{H}}

\def\z{\mathfrak z}

\def\so{\mathfrak s_{\overline 0}}
\def\ss1{\mathfrak s_{\overline 1}}

\def\hs1{\mathfrak h_{\overline 1}}

\def\G{\mathrm{G}}

\def\Zg{\mathrm{Z}}

\def\H{\mathrm{H}}

\def\Zg{\mathrm{Z}}
\def\Sg{\mathrm{S}}

\def\Bbb{\mathbb}

\def\H{\mathrm{H}}
\def\T{\mathrm{T}}

\def\Sp{\mathrm{Sp}}

\def\Og{\mathrm{O}}
\def\Ug{\mathrm{U}}

\def\Mg{\mathrm{M}}

\def \t{\tilde}
\def \wt{\widetilde}
\newcommand{\reg}[1]{ {#1}^{reg}}

\newcommand{\OP}{\mathop{\rm{OP}}}

\def\W{\mathsf{W}}
\def\Wv{\mathrm{W}}
\def\W+{\mathrm{W}_{\BB C}}
\def\Vv{\mathrm{V}}

\def\V{\mathsf{V}}

\def\Ssf{\mathsf{S}}

\def\DD{\mathbb{D}}
\def\HH{\mathbb{H}}
\def\ZZ{\mathbb{Z}}

\def\End{\mathop{\hbox{\rm End}}\nolimits}
\def\diag{\mathop{\hbox{\rm diag}}\nolimits}
\def\det{\mathop{\hbox{\rm det}}\nolimits}

\def\Ad{\mathop{\hbox{\rm Ad}}\nolimits}

\def\Re{\mathop{\hbox{\rm Re}}\nolimits}

\def\tr{\mathop{\rm tr}\nolimits}

\def\sgn{\mathop{\hbox{\rm sgn}}\nolimits}

\def\vol{\mathop{\hbox{\rm vol}}\nolimits}

\def\lim{\mathop{\hbox{\rm lim}}\nolimits}

\newcommand\inner[2]{\langle #1,#2\rangle}

\def\Oo{\mathcal{O}}

\def\Ss{\mathcal{S}}

\def\fontindex{\arabic}

\def\fonttitre{\textsf}
\newcounter{thh}

\newtheorem{thm}[thh]{\fonttitre{Theorem}}

\newtheorem{pro}[thh]{\fonttitre{Proposition}}
\newtheorem*{pro*}{\fonttitre{Proposition}}
\newtheorem{cor}[thh]{\fonttitre{Corollary}}
\newtheorem*{coro*}{\fonttitre{Corollary}}
\newtheorem{lem}[thh]{\fonttitre{Lemma}}
\theoremstyle{definition}
\newtheorem{rem}{\fonttitre{Remark}}

\newtheorem*{defi*}{\fonttitre{Définition}}

\newtheorem*{nota*}{\fonttitre{Notation}}
\def\muet{ \ifthenelse{\equal{a}{b}}}
\def\nn{\nonumber}
\newcommand{\thmlist}{
\renewcommand{\theenumi}{\alph{enumi}}
\renewcommand{\labelenumi}{(\theenumi)}}

\def\biblio{\sloppy
\bibliographystyle{alpha}
\bibliography{article}}

\author[M. McKee, A Pasquale and T. Przebinda]{M. McKee, A. Pasquale and T. Przebinda}
\title[Symmetry breaking operators for $(\Ug_l,\Ug_{l'})$]
{Symmetry breaking operators\\ for the reductive dual pair $(\Ug_l,\Ug_{l'})$}

\begin{document}
\thanks{The second author is grateful to the University of Oklahoma for hospitality and financial support. The third author gratefully acknowledges hospitality and financial support from the Universit\'e de Lorraine and partial support from the National Science Foundation under Grant DMS-2225892. Part of this research took place within the online Research Community on Representation Theory and
Noncommutative Geometry sponsored by the American Institute of Mathematics. 
}

\date{}

\maketitle
\begin{abstract}
We consider the dual pair $(\G,\G')=(\Ug_l,\Ug_{l'})$ in the symplectic group $\Sp_{2ll'}(\R)$. 
Fix a Weil representation of the metaplectic group $\wt{\Sp}_{2ll'}(\R)$.  Let
$\wt\G$ and $\wt\G'$ be the preimages of $\G$ and $\G'$ under the metaplectic cover 
$\wt{\Sp}_{2ll'}(\R)\to \Sp_{2ll'}(\R)$, and let $\Pi\otimes\Pi'$ be a genuine irreducible representation of $\wt\G\times\wt\G'$. We study the Weyl symbol $f_{\Pi\otimes\Pi'}$ of the (unique up to a possibly zero constant) symmetry breaking operator (SBO) intertwining the Weil representation with $\Pi\otimes\Pi'$. This SBO coincides with the orthogonal projection of the space of the Weil representation onto its $\Pi$-isotypic component and also with the orthogonal projection onto its $\Pi'$-isotypic component. Hence
$f_{\Pi\otimes\Pi'}$ can be computed in two different ways, one using $\Pi$ and the other using $\Pi'$.
By matching the results, we recover Weyl's theorem stating that $\Pi\otimes\Pi'$ occurs in the Weil representation with multiplicity at most one and we also recover the complete list of the representations $\Pi\otimes\Pi'$ occurring in Howe's correspondence. 

\end{abstract}
\tableofcontents
\section*{\bf Introduction}
\label{Introduction}
Let $\Wv$ be a finite dimensional vector space over $\Bbb R$ equipped with a non-degenerate symplectic form $\langle\cdot ,\cdot \rangle$, and let $\Sp(\Wv)$ and $\wt\Sp(\Wv)$  denote the corresponding symplectic and  metaplectic groups, respectively.
Fix an irreducible dual pair $\G, \G'\subseteq \Sp(\Wv)$ in the sense of Howe and let
$\wt\G$, $\wt\G'$ be the preimages of $\G$, $\G'$ in $\wt\Sp(\Wv)$. 
Let $\omega$ denote a fixed Weil representation of $\wt\Sp(\Wv)$ (up to unitary equivalence, there are two of them) and let $\omega^\infty$ be
the associated smooth representation. 
Consider two irreducible admissible representations $\Pi$ of $\wt\G$ and $\Pi'$ of $\wt\G'$ which are in Howe's correspondence, \cite{HoweTrans}. This means that $\Pi\otimes \Pi'$ is realized as a quotient of $\omega^\infty$. 
The space of operators intertwining $\omega^\infty$ and $\Pi\otimes\Pi'$ is one dimensional. In \cite{TKobayashiProgram}, any such intertwining operator is called a \textit{symmetry breaking operator} (SBO). The construction of SBOs is part of Kobayashi's program for branching problems in  representation theory of real reductive groups. 

As explained in \cite{PrzebindaUnitary}, the SBOs attached as above to Howe's correspondence are pseudodifferential operators in the Weyl calculus. Their symbols are $\G\G'$-invariant tempered distributions on $\Wv$, called \textit{intertwining distributions}. We denote by $f_{\Pi\otimes \Pi'}$ the (suitably normalized) intertwining distribution of the nonzero SBOs from $\omega^\infty$ to $\Pi\otimes\Pi'$. In \cite{McKeePasqualePrzebindaWCSymmetryBreaking} we explicitely computed $f_{\Pi\otimes \Pi'}$ under the assumption that one of the two members of the dual pair, $\G$, is compact.
The SBO is (a nonzero constant multiple of) the orthogonal projection of the space of 
$\omega^\infty$ onto its $\Pi$-isotypic component:
\begin{equation}
\label{omegaPi}
\omega(\check\Theta_\Pi)=\OP(\mathcal K( T(\check\Theta_\Pi)))\,,
\end{equation}
where $\Theta_\Pi$ is the character of $\Pi$, $\check{\Theta}_\Pi(g)=\Theta_\Pi(g^{-1})$ for all $g\in \G$, and $T$ is the embedding $\wt\Sp(\Wv)$ in $\Ss'(\Wv)$ attached to the Weil representation. Here $\Ss'(\Wv)$ is the space of tempered distributions on $\Wv$. 
As in \cite{McKeePasqualePrzebindaWCSymmetryBreaking}, we fix the normalization of the
intertwining distribution as follows: 
$$
f_{\Pi\otimes \Pi'}=\frac{1}{2}  T(\check{\Theta}_\Pi)=\frac{1}{2} \int_{\wt{\G}} \check{\Theta}_\Pi(\t{g})T(\t{g}) \, d\t{g}=\int_\G \check{\Theta}_\Pi(\t{g})T(\t{g}) \, dg\,,
$$
where the Haar measures on $\wt{\G}$ and $\G$ are normalized so that $\vol(\wt{\G})=2\vol(\G)$.
Knowing $f_{\Pi\otimes \Pi'}$ allowed us in \cite{McKeePasqualePrzebindaWC_WF}  to compute  the wavefront set of $\Pi'$ by elementary means. 
Knowing $f_{\Pi\otimes \Pi'}$ also allows us to show that no SBO associated with $\Pi\otimes \Pi'$ is a differential operator.
Indeed, as proved in \cite[Corollary 13]{McKeePasqualePrzebindaWCSymmetryBreaking}, the SBO associated with $\Pi\otimes \Pi'$ is a differential operator if and only if $f_{\Pi\otimes \Pi'}$ has support equal to $\{0\}$. In turn, this means that the wavefront set of $\Pi'$ is $\{0\}$, that is, $\Pi'$ is finite dimensional. By classification, see e.g. \cite[Appendix]{PrzebindaInfinitesimal}, all highest weight representations $\Pi'$ occurring in Howe's correspondence are infinite dimensional unless $\G'=\Ug_{l'}$. This case was left open. 
Studying the SBO for the dual pair $(\G,\G')=(\Ug_l,\Ug_{l'})$ is the topic of the present paper.
For this pair, the formula for $f_{\Pi\otimes \Pi'}$ found in 
\cite{McKeePasqualePrzebindaWCSymmetryBreaking} simplifies significantly. We prove in Proposition \ref{ul, ul' 3} that the intertwining distribution $f_{\Pi\otimes \Pi'}$ is actually a non-zero real analytic function on $\Wv$. In particular, its support is equal to $\Wv$.

Let $\Pi$ be an irreducible genuine representation of $\wt{\G}$. The distribution $T(\check{\Theta}_\Pi)$ depends only on $\Pi$ (no representation of $\wt{\G'}$ is involved). 
We find the conditions on $\Pi$ in terms of highest weights so that $T(\check{\Theta}_\Pi)\neq 0$, 
i. e. $\Pi$ occurs in $\omega$; see Corollary \ref{ul, ul' 2} and Propositions \ref{TThetaPi' l'>l} and \ref{TThetaPi' l=l'}.

Since $\G'=\Ug_{l'}$ is compact, we can reverse the roles of $\G$ and $\G'$ and compute
$T(\check{\Theta}_{\Pi'})$ for any irreducible genuine representation $\Pi'$ of $\wt{\G'}$. 
Suppose that $T(\check{\Theta}_\Pi)\neq 0$. In Corollary \ref{equality intertwining distributions}, we prove that there exists a unique $\Pi'$ and a nonzero constant $C_{\Pi\otimes \Pi'}$ such that 
$T(\check{\Theta}_\Pi)=C_{\Pi\otimes \Pi'} T(\check{\Theta}_{\Pi'})$. Equivalently, 
by \eqref{omegaPi}, 
$$
\omega(\check{\Theta}_\Pi)=C_{\Pi\otimes \Pi'} \omega(\check{\Theta}_{\Pi'})\,.
$$
Finally, we prove that $C_{\Pi\otimes \Pi'}=1$; see Lemma \ref{start constant is one} and Theorem \ref{constant is one}. Hence we recover Weyl's theorem: if both $\G$ and $\G'$ are compact then the restriction of the Weil representation to $\wt{\G}\times\wt{\G'}$ decomposes with multiplicity one. 
Furthermore, the description of Howe's correspondence for the dual pair $(\Ug_l,\Ug_{l'})$ follows.

Notice that there is one additional degenerate dual pair for which both members are compact, namely $(\Og_2^*,\Sp_{l'})$. It is not a ``dual pair'' because the centralizer of $\Og_2^*=\Ug_1$ in 
the symplectic group is isomorphic to $\Ug_{2l'}$, which contains $\Sp_{l'}$.
The case $(\G,\G')=(\Ug_1,\Ug_{2l'})$ suffices to understand the degenerate pair $(\Og_2^*,\Sp_{l'})$. Indeed, given the correspondence $\Pi \longleftrightarrow \Pi'$ for $(\Ug_1,\Ug_{2l'})$, by Howe's theory, one obtains the correspondence for the degenerate pair by restricting $\Pi'$ to $\Sp_{l'}$. The intertwining distributions 
for the degenerate pair are therefore the same as those for $(\Ug_1,\Ug_{2l'})$.

\section{\bf The dual pair $\G=\Ug_{l},\ \G'=\Ug_{l'}$, where $l\leq l'$}
\label{ul, ul'}\label{two unitary groups}
In this section we collect the properties of the dual pair $(\G,\G')=(\Ug_l,\Ug_{l'})$ that will be needed to compute the intertwining distributions. We assume without loss of generality that $l\leq l'$. 

\subsection{\bf Lie supergroup realization}

Consider $\Vv=\C^l$ endowed with its usual positive definite hermitian form $(\cdot,\cdot)$ preserved by $\G$, and $\Vv'=\C^{l'}$ endowed with the skew hermitian form $(\cdot,\cdot)'$, equal to $i$ times the usual hermitian form, preserved by $\G'$.  
We regard $\C^{l|l'}=\Vv\oplus\Vv'$ as $\ZZ/2\ZZ$-graded vector space over $\C$ with $\Vv$ as even part and $\Vv'$ as odd part. 
Let $\u_d$ denote the space of complex skew hermitian matrices of size $d\times d$. 
Set 
\begin{align}
\label{supergroupS}
&\Sg=\G\times \G' =\left\{ \begin{pmatrix}
g & 0\\
0 & g'
\end{pmatrix}; 
g\in \G, \, g'\in \G'\right\}\,,\\
\label{gg'}
&\g=\left\{
\begin{pmatrix}
x & 0\\
0 & 0
\end{pmatrix}; 
x\in \u_l
\right\}, \qquad
\g'=\left\{
\begin{pmatrix}
0 & 0\\
0 & x'
\end{pmatrix}; 
x'\in\frak u_{l'}
\right\}, \\
\label{ss}
&\so=\g\oplus\g', \qquad 
\ss1
=\left\{ 
\begin{pmatrix}
0 & w\\
w^* & 0
\end{pmatrix}; 
w\in M_{l,l'}(\C)
\right\}, \qquad \fraks=\so\oplus\ss1\,,
\end{align}
where $w^*=i\overline w^t$ and $\cdot^t$ denotes the transpose. 
Then $(\Sg,\fraks)$ is a real Lie supergroup, i.e. a real Lie group $\Sg$ together with a real Lie superalgebra $\fraks=\so\oplus \ss1$ whose even part $\so$ is the Lie algebra of $\Sg$.

Define $\Ssf\in M_{l+l'}(\C)$ by 
\begin{equation}
\label{S} 
\Ssf=\begin{pmatrix}
I_l & 0\\ 0 & -I_{l'}
\end{pmatrix},
\end{equation}
where $I_d$ denotes the $d\times d$ identity matrix. 
Notice that for $\begin{pmatrix}
x & w \\ w^* & x'
\end{pmatrix}, 
\begin{pmatrix}
y & w' \\ {w'}^* & y' 
\end{pmatrix} \in \fraks$, we have
$$
\tr\left( \Ssf \begin{pmatrix}
x & w \\ w^* & x'
\end{pmatrix}\begin{pmatrix}
y & w' \\ {w'}^* & y'
\end{pmatrix}\right)=\tr(xy)-\tr(x'y')+2\Re\tr(w{w'}^*)
\,.
$$
Define
\begin{equation}
\label{form on fraks}
\left\langle\begin{pmatrix}
x & w \\ w^* & x'
\end{pmatrix},\begin{pmatrix}
y & w' \\ {w'}^* & y'
\end{pmatrix}\right\rangle
=2\Re\tr\left( \Ssf \begin{pmatrix}
x & w \\ w^* & x'
\end{pmatrix}\begin{pmatrix}
y & w' \\ {w'}^* & y'
\end{pmatrix}\right)\,.
\end{equation}
(On the right-hand side, $2\Re\tr$ is sometimes written as $\tr_{\C/\R}$, that is the trace of a complex $d\times d$ matrix considered as a real $2d\times 2d$ matrix via the identification $\C \ni a+ib \to \begin{pmatrix}
a & b\\
-b & a
\end{pmatrix} \in M_2(\R)$.)

Restricted to $\ss1$, the form \eqref{form on fraks} defines a non-degenerate symplectic form.
Set $\Wv=M_{l,l'}(\C)$ considered as a vector space over $\R$. We endow $\Wv$ with the non-degenerate symplectic form $\inner{\cdot}{\cdot}$ obtained by the identification 
\begin{equation}
\label{identification ss1 and W}
\Wv\ni w\to \begin{pmatrix}
0 & w\\
w^* & 0
\end{pmatrix} \in\ss1\,, 
\end{equation}
i.e. we set 
\begin{equation}
\label{symplectic form}
\inner{w}{w'}=4\Re\tr(w{w'}^*) \qquad (w,w'\in \Wv)\,,
\end{equation}
The restrictions of \eqref{form on fraks} to $\g$ and $\g'$ yield two symmetric bilinear forms:
$2\Re\,\tr(xy)$ and $-2\Re\,\tr(xy)$. Define
\begin{align}
\label{symmetric form ulul'}
&B(x,y)=2\pi\,\Re\,\tr(xy) \qquad (x,y\in\g)\\
&B'(x',y')=-2\pi\,\Re\,\tr(x'y') \qquad (x',y'\in\g').\nn
\end{align}
Here and in the following, we will often identify $\g$ with $\u_l$ and $\g'$ with $\u_{l'}$ via 
\eqref{gg'}. In this context, the unnormalized moment maps are
\begin{equation}
\label{tau tau'}
\tau:\Wv \ni w\to ww^*\in \g \qquad \text{and} \qquad \tau':\Wv \ni w\to w^*w\in \g'\,.
\end{equation}
By \eqref{identification ss1 and W}, the action of 
$\Sg$ on $\ss1$ by conjugation corresponds to the action of $\G\times \G'$ on $\Wv$ given by 
\begin{equation}
\label{Ad-action}
(g,g').w=gw{g'}^{-1} \qquad (g\in \G, \, g'\in \G',\, w\in \Wv).
\end{equation}
\subsection{\bf Cartan subalgebras and Cartan subspaces}

Let $E_{i,j}\in \g_\C=M_l(\C)$ denote the matrix with all entries equal to 0 except that of index $(i,j)$, which is equal to $1$. Let $E'_{h,k}\in \g'_\C=M_{l'}(\C)$ be similarly defined. Set
\begin{equation}
\label{JjJ'j}
J_j=iE_{j,j} \quad (1\leq j\leq l), \qquad J'_j=iE'_{j,j} \quad (1\leq j\leq l').
\end{equation}

Up to a conjugation, there is only one Cartan subspace $\hs1$ in $\ss1$.
It consists of the matrices as in \eqref{ss} with
\begin{equation}\label{winW and hs1}
w=
\begin{pmatrix}
w_1 & 0 &\dots & 0 & 0 & \dots & 0\\
 0 & w_2 &\dots & 0 & 0 & \dots & 0 \\
\vdots & \vdots & \ddots & \vdots & \vdots  &  & \vdots \\
 0 & 0 &\dots & w_l & 0 & \dots & 0
\end{pmatrix}
\end{equation}
where $w_j\in \R$ for $1\leq j\leq l$. 
Let $\hs1^2\subseteq \so$ be the subspace spanned by all the squares of the elements 
$\begin{pmatrix}
0 & w\\
w^* & 0
\end{pmatrix} \in \hs1$.
Hence, $\hs1^2$ consists of the matrices
\begin{equation}\label{Cartansubspacesforu1opq1}
\diag(iy_1, iy_2,\dots, iy_l, iy_1, iy_2,\dots, iy_l,0,0,\dots,0)=
\left(\begin{array}{c|cc}
\sum_{j=1}^l y_j J_j & 0 & 0\\[.2em] 
\hline 
0 & \sum_{j=1}^l y_j J'_j & 0\\
0 & 0 & 0
\end{array}\right)
\end{equation}
where $y_1, y_2,\dots, y_l\in \R$.
Let $\h\subseteq \g$ and $\h'\subseteq \g'$ be the diagonal Cartan subalgebras. Then the map
\begin{equation}
\label{h1-tautau'}
\hs1\ni \begin{pmatrix}
0 & w\\
w^* & 0
\end{pmatrix} \to \begin{pmatrix}
0 & w\\
w^* & 0
\end{pmatrix}^2 \equiv(\tau(w), \tau'(w))\in \h\oplus\h'
\end{equation}
induces the embedding
\begin{equation}\label{embedding h into h'}
\h\ni \diag(iy_1, iy_2,\dots, iy_l)=\sum_{j=1}^l y_j J_j \to \diag(iy_1, iy_2,\dots, iy_l, 0,\dots, 0)=\sum_{j=1}^l y_j J'_j \in \h'.
\end{equation}
We shall identify $\h$ with its image in $\h'$. Then we have a direct sum decomposition
\begin{equation}\label{h oplus h''}
\h'=\h\oplus\h'',
\end{equation}
where $\h''$ consists of the matrices
$$
\diag(0, \dots,0, iy'_{l+1}, iy'_{l+2}, \dots, iy'_{l'})=\sum_{j=l+1}^{l'} y'_j J'_j
$$
where $y'_{l+1}, y'_{l+2},\dots, y'_{l'}\in \R$.

Let us introduce coordinates on $i\h^*$ and $i{\h'}^*$. Let $J^*_j$, $1\leq j\leq l$, be the basis of $\h^*$ which is dual to $J_j$, $1\leq j\leq l$ in the sense that $J^*_j(J_k)=\delta_{j,k}$ for all
$1\leq j,k\leq l$. Similarly, let  ${J'_j}^*$, $1\leq j\leq l'$, be the basis of ${\h'}^*$ which is dual to $J'_j$, $1\leq j'\leq l$. 
Set 
\begin{equation}
\label{ej}
e_j=-iJ^*_j \quad (1\leq j\leq l), \qquad e'_j=-i{J'}^*_j \quad (1\leq j\leq l')\,.
\end{equation}
If $\mu\in i\h^*$ and $\mu'\in i{\h'}^*$, then $\mu=\sum_{j=1}^l \mu_j e_j$ and 
$\mu'=\sum_{j=1}^{l'} \mu'_j e'_j$ with $\mu_j, \mu'_j\in \R$ for all $j$.

\subsection{\bf Root structures}

Let us fix the order of the roots of $(\g'_\C,\h'_\C)$ so that $\mu'\in i{\h'}^*$ is regular and dominant if and only if
\begin{equation}\label{regular dominant h'}
\mu'_1>\mu'_2>\dots>\mu'_{l'}.
\end{equation}
For the roots of $(\g_\C,\h_\C)$, we fix the order induced by the embedding of $\h$ in $\h'$. 
Hence $\mu\in i{\h}^*$ is regular and dominant if and only if
\begin{equation}\label{regular dominant h}
\mu_1>\mu_2>\dots>\mu_{l}.
\end{equation}
Let $\Delta^+=\{\alpha_{k,j}=e_j-e_k; 1\leq j<k\leq l\}$ be the corresponding set of positive roots of  
$(\g_\C,\h_\C)$. The root space corresponding to $\alpha_{k,j}$ is
$(\g_\C)_{\alpha_{k,j}}=\C E_{k,j}$.
Similarly, let ${\Delta'}^+=\{\alpha'_{k,j}=e'_j-e'_k; 1\leq j<k\leq l'\}$ be the set of positive roots for $(\g'_\C,\h'_\C)$. 
We denote by $\pi_{\g/\h}$ and $\pi_{\g'/\h'}$  the product of all the elements
of $\Delta^+$ and ${\Delta'}^+$, respectively.  
Moreover, let $\z'\subseteq \g'$ denote the centralizer of $\h\subseteq \h'$, and let $\pi_{\g'/\z'}$ be the product of the positive roots for which the root spaces do not occur in the complexification of $\z'$. Explicitly, 
\begin{align}
\label{products of roots g/h}
&\pi_{\g/\h}(y)=\prod_{1\leq j<k\leq l} i(-y_j+y_k) \qquad  (y=\sum_{j=1}^l y_jJ_j\in\h)\,, \\ 
\label{products of roots g'/h'}
&\pi_{\g'/\h'}(y')=\prod_{1\leq j<k\leq l'}i(-y'_j+y'_k) \qquad  (y'=\sum_{j=1}^{l'} y'_j J'_j\in\h')\,,\\ 
\label{products of roots g'/z'}
&\pi_{\g'/\z'}(y')=\pi_{\g/\h}(y')\cdot\prod_{1\leq j\leq l}(-iy'_j)^{l'-l}  \qquad  (y'=\sum_{j=1}^l y'_jJ'_j
\in\h \subseteq \h')\,. 
\end{align}
The following notation uses the identification \eqref{identification ss1 and W} between $\ss1$ and $\Wv$. Accordingly, $\hs1$ is identified with the subspace of matrices of the form \eqref{winW and hs1}.
If $\pi_{\so/\hs1^2}$ denotes the product of the positive roots of $\hs1^2$ in the complexification of  $\so=\g \oplus \g'$, then
\begin{equation}
\label{products of roots in so}
\pi_{\so/\hs1^2}(w^2)=\pi_{\g/\h}(\tau(w))\pi_{\g'/\z'}(\tau'(w)) \qquad (w\in\hs1).
\end{equation}
As one can see from \eqref{products of roots g/h} and \eqref{products of roots g'/z'}, 
\begin{equation}
\label{Ch1}
\pi_{\so/\hs1^2}(w^2)=C(\hs1) |\pi_{\so/\hs1^2}(w^2)|\,, \quad \text{where} \quad 
C(\hs1)=(-1)^{l(l'-l)} i^{l(l'-1)}\,.
\end{equation}
We set
\begin{align} 
\reg{\h}&=\{y \in \h; \;\pi_{\g/\h}(y)\neq 0\}=\{y \in \h; \; y_j\neq y_k\; (1\leq j< k\leq l)\}\,,
\label{hreg}
\\
\reg{\hs1}&=\{w \in \hs1; \;\pi_{\so/\hs1^2}(w^2)\neq 0\} \nn\\
&=\{w \in \hs1; \; w_j\neq \pm w_k\; (1\leq j< k\leq l),  \, w_j\neq 0 \; (1\leq j\leq l)\}
\label{h1reg}
\,.
\end{align}
The Weyl group $W(\G,\h)$ acts on $\h$ by permutation of the coordinates.  
The Weyl group $W(\Sg,\hs1)$ acts on $\hs1$ by sign changes and permutation of the coordinates; 
see \cite[(6.3)]{PrzebindaLocal}.
Let $\h^+$ denote the positive Weyl chamber for $\Delta^+$.
We fix the Weyl chamber $\hs1^+$ in $\hs1$ defined by $w_1>w_2>\dots>w_l>0$.
Then $\h^+$ and $\hs1^+$ are fundamental domains for the actions of $W(\G,\h)$ and $W(\Sg,\hs1)$, respectively. 
By 
\cite[Lemma 3.5]{McKeePasqualePrzebindaWCestimates}, 
$\tau(\reg{\hs1})$ is $W(\G,\h)$-invariant and its closure is $\h\cap \tau(\Wv)$. Hence
\begin{equation}
\label{hcaptauW}
\h\cap \tau(\Wv)=\Big\{\sum_{j=1}^ly_jJ_j; \ \text{$y_j\geq 0$ for all $1\leq j\leq l$}\Big\}=
\tau(\hs1)
\end{equation} 
is a $W(\G,\h)$-invariant domain. 

\subsection{\bf Genuine representations}
\label{subsection:genuine}
Let $\ZZ_2$ be the two element kernel of the metaplectic cover $\wt{\Sp}(\Wv) \to \Sp(\Wv)$. 
A representation $\Pi$ of $\wt\G$ is called genuine provided its restriction to 
$\ZZ_2$ is a multiple of the unique non-trivial character of $\ZZ_2$. A representation of $\wt\G$
which occurs in the restriction of the Weil representation $\omega$ must be
genuine (but in general not all genuine representations of $\wt\G$ occurs in the restriction of $\omega$ to 
$\wt\G$). 

For the dual pair  $(\G,\G')=(\Ug_l,\Ug_{l'})$ and with the above choice of ordering in $i\h^*$, an irreducible representation of $\Pi$ of $\wt\G$ is genuine if and only if its highest weight is of the form
$\lambda=\sum_{j=1} \lambda_j e_j$ where
\begin{equation}
\label{lambda}
\lambda_j=\frac{l'}{2}+\nu_j,\qquad \nu_j\in \ZZ, \qquad \nu_1\geq \nu_2\geq \cdots \geq \nu_l\,.
\end{equation}
By definition, the  Harish-Chandra parameter of $\Pi$ is $\mu=\lambda+\rho$, where 
\begin{equation}
\label{rho}
\rho=\sum_{j=1}^l \Big(\frac{l+1}{2}-j\Big) e_j
\end{equation}
is half the sum of the positive roots of $\Delta^+$.
The restriction to $\wt{\G}$ of the metaplectic cover, $\wt\G \to \G$,  splits if and only if $l'$ is even;
see for instance \cite[Proposition D.8]{McKeePasqualePrzebindaWCSymmetryBreaking}. In this case, 
$\Pi$ factors to a representation of $\G$. 

Similar properties, with the role of $l$ and $l'$ reversed, hold for the genuine irreducible representations of $\wt{\G'}$.

\subsection{\bf Compatible positive complex structures}

Fix the compatible positive complex structure $J=-i1_\Wv$ on $\Wv$, where $1_\Wv$ indicates the identity map. The corresponding symmetric bilinear form on $\Wv$ is $4$ times the real part of the canonical hermitian form: 
\begin{equation}
\label{Jform}
\langle J(w),w'\rangle=4 \Re\tr(w\overline{w'}^t)
\qquad (w,w'\in \Wv)\,. 
\end{equation}

Let $I_d$ denote the $d\times d$ identity matrix and recall the action \eqref{Ad-action} of $\Sg$ on $\Wv$. Then $J$ can be realized as
\begin{equation}
\label{JinG}
J_\g=-iI_l=-i\sum_{j=1}^l E_{j,j}=-\sum_{j=1}^l J_j \in \g
\end{equation}
acting on $\Wv$ by left multiplication, or as
\begin{equation}
\label{JinG'}
J_{\g'}=iI_{l'}=i\sum_{j=1}^{l'} E'_{j,j}=\sum_{j=1}^{l'} J'_j \in \g'
\end{equation}
acting on $\Wv$ is by minus the right multiplication.

Recall the symmetric bilinear forms $B$ and $B'$ from \eqref{symmetric form ulul'}. Let $y=
\sum_{j=1}^l y_j J_j=\tau(w)$ with $w\in \hs1$. Then $-B(J_j,y)=2\pi y_j$. 
Hence
\begin{equation}
\label{sum yj and J}
2\pi \sum_{j=1}^l y_j=-\sum_{j=1}^l B(J_j,y)=B(J_\g,\tau(w))=2\pi \Re\tr(J_{\g}ww^*)=\frac{\pi}{2} 
\langle J(w),w\rangle\,.
\end{equation}
Similarly, let $y=\sum_{j=1}^l y_j J'_j=\tau'(w)$ with $w\in \hs1$. Then $B'(J'_j,y)=2\pi y_j$. 
Hence
\begin{align}
\label{sum yj and J'}
2\pi \sum_{j=1}^l y_j&=\sum_{j=1}^l B'(J'_j,y)=B'(J_{\g'},\tau'(w))=2\pi \Re\tr(J_{\g'}w^*w)=
-2\pi \Re\tr(w(J_{\g'}w)^*) \nn\\
&=-\frac{\pi}{2} \langle w,J(w)\rangle=\frac{\pi}{2} \langle J(w),w\rangle\,.
\end{align}

\subsection{\bf The Cayley transform}
\label{subsection:Cayley}
Let $\sp(\Wv)$ be the Lie algebra of $\Sp(\Wv)$. Set
\begin{eqnarray}
\label{spc}
&&\sp(\Wv)^c=\{x\in\sp(\Wv);\ x-1\ \text{is invertible}\}\,,\\
\label{Spc}
&&\Sp(\Wv)^c=\{g\in\Sp(\Wv);\ g-1\ \text{is invertible}\}\,.
\end{eqnarray}
The Cayley transform 
$c:\sp(\Wv)^c  \to \Sp(\Wv)^c$ is the bijective rational map defined by $c(x)=(x+1)(x-1)^{-1}$. Its inverse $c^{-1}:\Sp(\Wv)^c\to \sp(\Wv)^c$ is given by the same formula, 
$c^{-1}(g)=(g+1)(g-1)^{-1}$. 

If $(\G,\G')$ is a reductive dual pair with $\G$ compact, then all eigenvalues of $x \in \g\subseteq \End(\Wv)$ are purely imaginary. So $x-1$ is invertible. Hence $\g\subseteq \sp(\Wv)^c$ and the restriction of $c$ to $\g$ is analytic. One can show that, if $\G$ is a unitary group, then the range $c(\g)$ is $\{g\in\G;\ \det(g-1)\ne 0\}$, a connected dense subset of $\G$ containing $-1=c(0)$. 

Let us consider the restriction of $c$ to $\g$ as a map $c:\g\to \G$ and fix a real analytic lift $\t c:\g\to\wt\G$ of $c$. 
Set $c_-(x)=c(x) c(0)^{-1}$ and $\t c_-(x)=\t c(x) \t c(0)^{-1}$. Then $c_-(0)=1$ and 
$\t c_-(0)$ is the identity of the group $\wt\Sp(\Wv)$.

\subsection{\bf Normalization of measures}
\label{subsection:normalizations}
We normalize the Lebesgue measure $\mu_\Wv$ on $\Wv$ so that the volume of the unit cube with respect to the form $\langle J \cdot,\cdot\rangle$ on $\Wv$ is $1$. The same normalization applies to each subspace of $\Wv$, such as the Cartan subspace $\hs1$ of $\ss1$ defined in \eqref{winW and hs1}. We shall employ the usual notation $dw$ for $d\mu_\Wv(w)$. 

We now consider $\G=\Ug_n$ and $\g=\u_n$.
On $\g$, we fix the $\G$-invariant (real) inner product 
\begin{equation}
\label{inner product on un}
\kappa(x,y)=\Re\tr(x\overline{y^t})=-\Re\tr(xy)\,.
\end{equation}
Set $d=n^2$. Let $\{X_j; 1\leq j\leq d\}$ be a basis of $\g$ and let 
$\{X^*_j; 1\leq j\leq d\}$ be the dual basis of $\g^*$. Then 
\begin{equation}
\label{Omega1}
\big( \det(\kappa(X_j,X_k)_{j,k=1}^d)\big)^{1/2} X_1^* \wedge \dots \wedge X_d^*
\end{equation}
is an alternating $d$-form on $\g$ that is independent of the basis of $\g$ 
with the same orientation as the ordered basis
$\{X_1,\dots, X_d\}$. It defines a Lebesgue measure on $\g$ which we will denote by $dx$.
As basis $\{X_j; 1\leq j\leq d\}$ of $\g$, we fix 
\begin{multline}
\label{basis un}
\{J_j; 1\leq j \leq n\} \cup \big\{Y_{j,k}=\frac{1}{\sqrt{2}} (E_{j,k}-E_{k,j}); 1\leq j<k\leq n\big\} \\ \cup 
\big\{Y_{k,j}=\frac{i}{\sqrt{2}} (E_{j,k}+E_{k,j}); 1\leq j<k\leq n\big\}\,,
\end{multline}
which is orthonormal with respect to $\kappa$. 
For $X\in \g$, let $\wt{X}$ denote the corresponding left-invariant vector field on $\G$. 
Then the equations $\wt{X}^*_j(\wt{X}_k)=\delta_{j,k}$ uniquely determine  $d$ left-invariant $1$-forms $\wt{X}^*_j$ on $\G$. The measure on $\G$ associated with the left-invariant $d$-form
\begin{equation}
\label{Omega}
\Omega=\big( \det(\kappa(X_j,X_k)_{j,k=1}^d)\big)^{1/2} \wt{X}^*_1 \wedge \dots \wedge \wt{X}^*_d
\end{equation}
is a Haar measure on $\G$, which we will denote by $\mu$ or simply by $dg$. Since $\Omega$ at the identity $1\in \G$ coincides with \eqref{Omega1}, we will write that $(dg)_1=dx$. The Haar measure on 
the preimage $\wt{\G}$ of $\G$ in the metaplectic group will be normalized so that 
$\vol(\wt{\G})=2\vol(\G)$.

The restriction of $\kappa$ to the subspace $\h$ of diagonal matrices in $\g$ is an inner product on 
$\h$. As above, this fixes a Haar mesure $dh$ on $\H$ and Lebesgue measure $dy$ on $\h$ such that 
$(dh)_1=dy$. Notice that the choice of the basis \eqref{basis un} also fixes the orthonormal basis 
$\{J_j; 1\leq j \leq n\}$ on $\h$. Hence, if $(y_1,\dots,y_n)$ are the coordinates on $\h$ with respect to this basis, we have $dy=dy_1\cdots dy_n$. 

We fix on $\G/\H$ the quotient measure, i.e. the unique $\G$-invariant measure $d(g\H)$ such that
\begin{equation}
\label{quotient measure}
\int_\G f(g) \, dg=\int_{\G/\H}  \Big(\int_\H f(gh) \, d(g\H)\Big)\, dh \qquad (f\in C(\G))\,.
\end{equation}

With these normalizations, Weyl's integration formula on $\g$ states the following (see e.g. \cite[Corollary 3.14.2]{Duistermaat-Kolk} and \cite[pp. 94--95]{Macdonald}): if $f$ is a $\G$-invariant integrable function on $\g$, then 
\begin{equation}
\label{Weil on g}
\int_\g f(x)\, dx= \frac{c_{\rm Weyl}}{|W(\G,\h)|}  \int_\h f(y) |\pi_{\g/\h}(y)|^2 \, dy\,,
\end{equation} 
where $|W(\G,\h)|$ is the cardinality of the Weyl group, $\pi_{\g/\h}$ is as in \eqref{products of roots g/h}, and $c_{\rm Weyl}$ is a constant. To make this constant explicit, we need to fix an inner product $\inner{\cdot}{\cdot}$ on $i\h^*$. 
For $\lambda \in i\h^*$ let $y_\lambda \in \h$ be the unique element such that $\kappa(y_\lambda, y)=i\lambda(y)$ for every $y\in\h$. Define 
\begin{equation}
\label{kappa on ih*}
\inner{\lambda}{\mu}=\kappa(y_\lambda,y_\mu) \qquad (\lambda,\mu\in i\h^*)\,.
\end{equation} 
In particular, $y_{e_j}=J_j$ because $\kappa(J_k,y_{e_j})=ie_j(J_k)=\delta_{j,k}$ for $1\leq k\leq n$.
Hence $\inner{e_j}{e_k}=\kappa(J_j,J_k)=\delta_{j,k}$. Moreover, for every $\alpha_{k,j}=e_j-e_k\in \Delta$, we have $y_{\alpha_{k,j}}=J_j-J_k$. Let $\rho$ be as in \eqref{rho}.
Then 
$$
\inner{\rho}{\alpha_{k,j}}=\sum_{r=1}^n \big(\frac{n+1}{2}-r\big) \inner{e_r}{e_j-e_k}=k-j\,.
$$
Therefore
\begin{equation}
\label{prod rho alpha}
\prod_{\alpha\in \Delta^+} \inner{\rho}{\alpha}=\prod_{1\leq j<k\leq n} (k-j)=\prod_{j=1}^{n-1} j!
\end{equation}
Finally
\begin{equation}
\label{cWeyl}
c_{\rm Weyl}=\frac{\vol(\G)}{\vol(\H)}=\frac{(2\pi)^{|\Delta^+|}}{\prod_{\alpha\in \Delta^+} \inner{\rho}{\alpha}}=\frac{(2\pi)^{n(n-1)/2}}{\prod_{j=1}^{n-1} j!}
\end{equation}
The kernel of the exponential map $\exp:\h\to \H$ is $2\pi \h_\ZZ$, where 
$\h_\ZZ={\rm span}_\ZZ \{J_j; 1\leq j\leq n\}$. Our normalization of the Lebesgue measure on $\h$ is such that $\vol(\h/\h_\ZZ)=1$. Hence $\vol(\H)=\vol(\h/2\pi \h_\ZZ)=(2\pi)^n \vol(\h/\h_\ZZ)$ and
\eqref{cWeyl} lead to
\begin{equation}
\label{volumes H and G}
\vol(\H)=(2\pi)^n \qquad \text{and}
\qquad \vol(\G)=\frac{(2\pi)^{n(n+1)/2}}{\prod_{j=1}^{n-1} j!}
\end{equation}
Here and in the following, products on the empty set are by definition equal to 1. 

Set
\begin{equation}
\label{ch}
\ch(x)=\det(1-x)^{1/2}_\R \qquad (x\in \g)\,,
\end{equation}
where the subscript $\R$ indicates that we are viewing $\g$ as a vector space over $\R$. 
Then $\ch$ is $\G$-invariant and 
\begin{equation}
\label{ch on h}
\ch(y)=\prod_{j=1}^n (1+y_j^2)^{1/2} \qquad (y=\sum_{j=1}^n y_j J_j\in \h)\,.
\end{equation}
Recall from subsection \ref{subsection:Cayley} the Cayley transform $c:\g\to \G$. 

\begin{lem}
\label{Cg}
With the fixed normalizations of the measures on $\G$ and $\g$, for every integrable function 
$f$ on $\G$, 
\begin{equation}
\label{constant Haar measures on G}
\int_\G f(g) \, dg = \int_\g f(c(x)) j_\g(x) \, dx\,,
\end{equation}
where $j_\g(x)=2^{n^2}\ch^{-2n}(x)$.
\end{lem}
\begin{proof}
By Lemma \ref{lemma:Haar via Cayley}, the integral $f\to \int_\g f(c(x)) j_\g(x) \, dx$
defines a Haar measure $\mu_\G$ on $\G$. Hence $\mu=C_\g \mu_\G$ for some positive constant $C_\g$. To determine $C_\g$, we evaluate both sides of 
\eqref{constant Haar measures on G} for $f=1$. 

By Weyl's integration formula on $\g$, see \eqref{Weil on g}, and by \eqref{ch on h}, 
\begin{align*}
\int_\g \ch^{-2n}(x) \, dx&=\frac{1}{|W(\G,\h)|} \, \frac{\mu(\G)}{\mu(\H)} 
\int_\h \ch^{-2n}(y) |\pi_{\g/\h}(y)|^2\, dy\\
&=\frac{1}{|W(\G,\h)|} \, \frac{\mu(\G)}{\mu(\H)} 
\int_{\R^n} \prod_{j=1}^n (1+y^2)^{-n} \prod_{1\leq j<k\leq n} (y_j-y_k)^2 \, dy_1\cdots dy_n\,.
\end{align*}
The last integral is computed in \cite[(1.19) with $\alpha=\beta=n$ and $\gamma=1$]{Forrester-Warnaar}. 
It is equal to $(2\pi)^n 2^{-n^2} n!$. Using that $\mu(\H)=(2\pi)^n$, we conclude that
$$
\int_\g \ch^{-2n}(x) \, dx=2^{-n^2} \mu(\G)\,.
$$
Thus $C_\g=2^{n^2}$.
\end{proof}

\begin{rem}
As shown e.g. in \cite[Appendix B]{McKeePasqualePrzebindaWCSymmetryBreaking}, the function $j_\g$ is the Jacobian of the Cayley transform $c:\g\to \G$. Our normalization of the Haar measure on $\G$ is therefore so that it agrees with the pushforward of the fixed measure $dx$ on $\g$ via the Cayley transform. 
\end{rem}

We now come back to the notation $\G=\Ug_l$, $\G'=\Ug_{l'}$, and $\Sg=\G\times \G'$. Set 
$$
\Delta(\H)=\Big\{ \begin{pmatrix} 
h & 0 \\
0 & h
\end{pmatrix}\in \Mg_{2l}(\C); h\in \H\Big\}.
$$
Then centralizer of $\hs1$ in $\Sg$ is 
\begin{equation}
\label{S^hs1}
\Sg^{\hs1}
=\Delta(\H)\times\Ug_{l'-l}\,, 
\end{equation}
where the right-hand side is diagonally embedded in $\Mg_{l+l'}(\C)$.

We fix on $\mathfrak{s}=\g\times\g'$ the product Lebesgue measure of those
fixed above for $\g$ and $\g'$. Similarly, we consider on $\Sg=\G\times\G'$ the Haar measure $ds=dg\,dg'$ which is product of the Haar measures fixed on $\G$ and $\G'$. Let $\kappa$ and $\kappa'$ respectively denote the inner products on $\g$ and $\g'$ given by \eqref{inner product on un}. 
The restriction of $\kappa\times \kappa'$ to the Lie algebra $\Delta(\h)$ of $\Delta(\H)$ is 
$$
\kappa_{\Delta(\h)}((x,x),(y,y))=-2\Re\tr(xy) \qquad (x,y\in \h)\,.
$$
This inner product fixes a Lebesgue measure on $\Delta(\h)$ and an associated Haar measure 
on $\Delta(\H)$. Then $\vol(\Delta(\H))=2^{l/2} \vol(\H)=2^{l/2} (2\pi)^l$.
We choose the Haar measure $d_{\Sg^{\hs1}}s$ on $\Sg^{\hs1}$ which is the product of the fixed Haar measures on $\Delta(\H)$ and on $\Ug_{l'-l}$.
Thus 
\begin{equation}
\label{volume Shs1}
\vol(\Sg^{\hs1})=\vol(\Delta(\H))\vol(\Ug_{l'-l})=2^{l/2} \,\frac{(2\pi)^{(l'-l)(l'-l+1)/2+l}}{\prod_{j=1}^{l'-l-1} j!}\,.
\end{equation}
Finally, we endow $\Sg/\Sg^{\hs1}$ with the quotient measure, i.e. the unique $\Sg$-invariant measure $d(s\Sg^{\hs1})$ such that
\begin{equation}
\label{quotient measure on S}
\int_\Sg f(s) \, ds=\int_{\Sg/\Sg^{\hs1}} \Big(\int_{\Sg^{\hs1}}  f(st) \, d_{\Sg^{\hs1}}t\Big)\, 
 d(s\Sg^{\hs1}) \qquad (f\in C(\Sg))\,.
\end{equation}

\subsection{\bf Orbital integrals}
\label{subsection:orbital integrals}

Let $\Ss'(\Wv)^\Sg$ denote the space of $\Sg$-invariant tempered distributions on $\Wv$, where the $\Sg$-action is given by \eqref{Ad-action}. 
For $w\in \reg{\hs1}$, the orbital integral attached to the orbit $\mathcal{O}(w)=\Sg.w$ is the element of $\Ss'(\Wv)^\Sg$ defined for $\phi\in \Ss(\Wv)$ by 
$$
\mu_{\mathcal{O}(w),\hs1}(\phi)=
\int_{\Sg/\Sg^{\hs1}} \phi(s.w)\, d(s\Sg^{\hs1})
=\frac{1}{\vol(\Sg^{\hs1})} 
\int_{\Sg} \phi(s.w)\, ds\,.
$$

By Weyl--Harish-Chandra integration formula on $\Wv$, \cite[Theorem 21]{McKeePasqualePrzebindaSuper}, 
\begin{equation}
\label{weyl int on w 1}
\mu_\Wv(\phi)=\int_{\tau(\hs1^+)}|\pi_{\so/\hs1^2}(w^2)|\mu_{\Oo(w),\hs1}(\phi)\,d\tau(w)  \qquad (\phi\in \Ss(\Wv))\,,
\end{equation}
where $d\tau(w)$ is a normalization of the Lebsegue measure on $\h$ so that the equality \eqref{weyl int on w 1} holds.
Hence  
\begin{equation}
\label{CW}
d\tau(w)=C_\Wv dy_1\cdots dy_l\,,
\end{equation}
where $dy_1\cdots dy_l$ is the Lebesgue measure on $\h$ we fixed in subsection \ref{subsection:normalizations}. 
The constant $C_\Wv$ will be computed in Lemma \ref{the constant for weyl integration on w moved}.


Since, up to conjugates, $\hs1$ is the unique Cartan subspace in $\Wv$ and $l\leq l'$, the Harish-Chandra regular elliptic orbital integrals on $\Wv$ is the function
\begin{equation}
\label{HC orbital integral on W}
F:\tau(\reg{\hs1})\to \Ss'(\Wv)^\Sg
\end{equation}
defined by 
\begin{equation}
\label{HC orbital integral on W l<=l'}
F(y)=C_{\hs1}\pi_{\g'/\z'}(y)\,\mu_{\Oo(w),\hs1} \qquad (y=\tau(w)=\tau'(w),\; w\in \reg{\hs1})\,,
\end{equation}
where, by \eqref{Ch1},
$$
C_{\hs1}=C(\hs1)\,i^{\dim(\g/\h)}=(-1)^{l(l'-1)}i^{ll'}\,.
$$
Following Harish-Chandra's notation, we shall write $F_\phi(y)$ for $F(y)(\phi)$, where $\phi\in\Ss(\Wv)$.
In \cite[Theorem 3.6]{McKeePasqualePrzebindaWCestimates} we proved that $F$ has a unique 
extension from \eqref{HC orbital integral on W} to 
\begin{equation}\label{eq:extendedf1}
F:\h\to \Ss^*(\Wv)^\Sg
\end{equation}
which is supported in $\h\cap \tau(\Wv)$ and is $W(\G,\h)$-skew-invariant:
\begin{equation}\label{extension by the symmetry condition1}
F(sy)=\sgn_{\g/\h}(s) F(y) \qquad (s\in W(\G,\h),\ y\in \h)\,.
\end{equation}
The map \eqref{eq:extendedf1} is smooth on the subset where each $y_j\ne 0$ and, for any multi-index $\alpha=(\alpha_1,\dots,\alpha_l)$ with
\begin{equation}
\label{condition on multiindex}
\max(\alpha_1,\dots,\alpha_l)\leq l'-l-1
\end{equation}
the function  $\partial(J_1^{\alpha_1}J_2^{\alpha_2}\dots J_l^{\alpha_l})F(y)$ extends to a continuous function on $\h\cap\tau(\Wv)$ vanishing on the boundary of $\h\cap\tau(\Wv)$. 

The computation of $T(\check{\Theta}_\Pi)$ in section \ref{section:intertwining_distributions} uses the definition of $F$ for $l\leq l'$, as above. On the other hand, to compute $T(\check{\Theta}_{\Pi'})$, we need to consider the fact that $\Ug_{l'}$ is compact and hence work with the pair $(\Ug_{l'},\Ug_l)$ and the rank relation $l'\geq l$.

When $l'>l$, the Harish-Chandra regular elliptic orbital integrals on $\Wv$ is function 
$F':\tau(\reg{\hs1}) \to \Ss'(\Wv)^\Sg$ defined by 
\begin{equation}
\label{HC orbital integral on W l>l'}
F'(y)=C'_{\hs1}\pi_{\g/\h}(y)\,\mu_{\Oo(w),\hs1} \qquad (y=\tau(w),\; w\in \reg{\hs1})\,,
\end{equation}
where 
$$
C'_{\hs1}=C(\hs1)\,i^{\dim(\g'/\h')}  \quad\text{and}\quad \dim(\g'/\h')=l'(l'-1)\,.
$$
(In our notation, $F$ is attached to $\G$ and $F'$ is attached to $\G'$. The latter should not be confused with a possible derivation of $F$.) 
By \cite[Theorem 3.4]{McKeePasqualePrzebindaWCestimates},
$F'$ extends to a distribution-valued map
on $\reg{\h}\cap \tau(\Wv)$ that is $W(\G,\h)$-skew-invariant and smooth
in the sense that it is differentiable in the interior of $\reg{\h}\cap \tau(\Wv)$ and any derivative of 
$F'$ extends to a continuous function on the closure in $\h$ of any connected component of 
$\reg{\h}\cap \tau(\Wv)$.


\section{\bf The intertwining distributions}
\label{section:intertwining_distributions}

In this section we compute the tempered distributions $T(\check{\Theta}_\Pi)$ and $T(\check{\Theta}_{\Pi'})$ for genuine representations $\Pi$ of $\wt\G$ and $\Pi'$ of $\wt{\G'}$. These results are  refiniments of those obtained in \cite{McKeePasqualePrzebindaWCSymmetryBreaking} in the case of the pair $(\Ug_l,\Ug_{p,q})$ with $p=l'$ and $q=0$.

For two integers $a, b\in \ZZ$ define the following polynomial in the real variable $\xi$,
\begin{equation}\label{Pab2}
P_{a,b,2}(\xi)=\begin{cases}
\sum_{k=0}^{b-1}\frac{a(a+1)\cdots(a+k-1)}{k!(b-1-k)!}2^{-a-k}\xi^{b-1-k}&\text{if}\ b\geq 1\\
0&\text{if}\ b\leq 0,
\end{cases}
\,
\end{equation}
where $a(a+1)\cdots(a+k-1)=1$ if $k=0$. 
We also set
\begin{align}
\label{Pab-2}
P_{a,b,-2}(\xi)&=P_{b,a,2}(-\xi) \qquad (\xi\in\R,\ a,b\in\ZZ)\,,\\
\label{Pab}
P_{a,b}(\xi)&=
2\pi \big( P_{a,b,2}(\xi)\mathbb{I}_{\R^+}(\xi)+P_{a,b,-2}(\xi)\mathbb{I}_{\R^-}(\xi)\big) \qquad (\xi\in\R,\ a,b\in\ZZ)\,,
\end{align}
where $\mathbb{I}_{A}$ denotes the indicator function of the set $A$.
Notice that, if $b\geq 1$, then 
\begin{align}
\label{Pab2-confluent}
P_{a,b,2}(\xi)
&=(-1)^{b-1} 2^{-a-b+1} \frac{\Gamma(-a+1)}{\Gamma(-a-b+2) \, (b-1)!} 
\, {}_1 F_1 \big(-b+1;-a-b+2;2\xi\big) \nn\\
&=(-1)^{b-1} 2^{-a-b+1} L^{-a-b+1}_{b-1}(2\xi)\,,
\end{align}
where $\Gamma$ is the gamma function, 
${}_1 F_1$ is the confluent hypergeometric function, and $L^\alpha_n(x)$ 
is a Laguerre polynomial. See \cite[6.9(36), \S 10.12]{Erdelyi}.

Set
\begin{equation}\label{delta-beta}
\delta=\frac{l'-l+1}{2} \qquad \text{and} \qquad \beta=2\pi\,.
\end{equation}
Let $\mu\in i\h^*$ be the Harish-Chandra parameter of $\Pi$
and define
\begin{equation}
\label{eq:ajbj}
\begin{array}{lll}
a_j=-\mu_j-\delta+1, & b_j=\mu_j-\delta+1 \quad &(1\leq j\leq l)\,,\\
a_{s,j}=-(s\mu)_j-\delta+1,\qquad \quad&b_{s,j}=(s\mu)_j-\delta+1 \qquad\quad &(s\in W(\G,\h), 1\leq j\leq l)\,.
\end{array}
\end{equation}
Observe that $a_j=a_{1,j}$ and $b_j=b_{1,j}$. Moreover, by \eqref{lambda}, \eqref{rho} and \eqref{delta-beta}, all the $a_{s,j}$ and $b_{s,j}$ are integers. 

\begin{lem}
\label{ul, ul' 1}
There is a nonzero constant $C$ (which depends only of the dual pair $(\Ug_l,\Ug_{l'})$), 
such that for $\phi\in \Ss(\Wv)$
\begin{align}\label{l<=l' and both compact 1}
T(\check\Theta_\Pi)(\phi)&=C \check{\chi}_\Pi(\t{c}(0)) 
\int_{\h\cap\tau(\Wv)} \prod_{j=1}^l  P_{a_j,b_j,2}(\beta y_j)e^{-\beta y_j}\cdot F_\phi(y)\,dy\\
&=\frac{C}{| W(\G,\h)|} \, \check{\chi}_\Pi(\t{c}(0))  \sum_{s\in W(\G,\h)} \sgn(s)
\int_{\h\cap\tau(\Wv)}
\prod_{j=1}^l  P_{a_{s,j},b_{s,j},2}(\beta y_j)e^{-\beta y_j}\cdot F_\phi(y)\,dy, \nn
\end{align}
where $\chi_\Pi$ is the central character of $\Pi$ and
$\t{c}$ is the fixed real analytic lift of the Cayley transform; see subsection \ref{subsection:Cayley}.
\end{lem}
\begin{proof} 
By the normalization of the Haar measure on $\wt{\G}$, 
$$
T(\check\Theta_\Pi)(\phi)=2\int_{\G}\check\Theta_\Pi(\t g) T(\t g)(\phi)\,dg.
$$
Since $\G=-\G^0$, by \cite[Theorem 4]{McKeePasqualePrzebindaWCSymmetryBreaking}, 
there is a non-zero constant $C_0$ (which depends only of the dual pair)
such that for all $\phi\in\Ss(\Wv)$
\begin{equation}\label{main thm for l<l' a}
\int_{\G}\check\Theta_\Pi(\t g) T(\t g)(\phi)\,dg=C_0 \,
\check{\chi}_\Pi(\t{c}(0)) 
\int_{\h\cap\tau(\Wv)}
\left(\prod_{j=1}^l  \left(p_j(y_j) +q_j(-\partial(J_j))\delta_0(y_j)\right)\right)\cdot F_\phi(y)\,dy\,,
\end{equation}
where for all 
$1\leq j\leq l$
$$
p_j(y_j)=P_{a_j,b_j}(\beta y_j)e^{-\beta|y_j|}
\qquad (y_j \in \R)\,,
$$
$q_j$ is a polynomial in $y_j$ of degree $-a_j-b_j=2\delta-2=l'-l-1$, and
$\delta_0$ is the Dirac delta at $0$. 
Because of the description \eqref{hcaptauW} of $\h\cap\tau(\Wv)$, in the integrand, we can replace $p_j(y_j)$ with 
$
\beta P_{a_j,b_j,2}(\beta y_j)e^{-\beta y_j}\,. 
$
Moreover, 
\begin{multline*}
\Big( \prod_{j=1}^l  \big(p_j(y_j) +q_j(-\partial(J_j))\delta_0(y_j)\big) \Big) \cdot F_\phi(y)\\
=\sum_{\gamma\subseteq\{1, 2, \dots, l\}} \beta^{l-|\gamma|}
\prod_{j\notin \gamma}  P_{a_j,b_j,2}(\beta y_j)e^{-\beta y_j} \cdot \prod_{j\in\gamma} 
q_j(-\partial(J_j))\delta_0(y_j) \cdot F_\phi(y)\,.
\end{multline*}
Let $\gamma=\{j_1,\dots,j_k\}\neq \emptyset$. Integration over the variables corresponding to $\gamma$ yields
\begin{multline*}
\int_{y_{j_1}\geq 0}\cdots\int_{y_{j_k}\geq 0} 
\Big( \prod_{h=1}^k  
 q_{j_h}(-\partial(J_{j_h}))\delta_0(y_{j_h}) \Big) \cdot F_\phi(y) \; dy_{j_1}\cdots dy_{j_k}\\
=
\Big[\Big( \prod_{h=1}^k  
 q_{j_h}(-\partial(J_{j_h}))\Big) \cdot F_\phi(y) \Big]_{y_{j_1}=0,\dots,y_{j_k}=0}\,,
\end{multline*}
which vanishes by \eqref{condition on multiindex}, since the degree of the polynomials $q_{j_h}$ is equal to $l'-l-1$.
Thus, the integrals over $\h\cap\tau(\Wv)$ of all terms with $\gamma\ne \emptyset$ are zero, and the first equality in \eqref{l<=l' and both compact 1} follows.

For the second equality, notice that, since $(sy)_j=y_{s^{-1}(j)}$, we have
\begin{align*}
\prod_{j=1}^l  P_{a_{s,j},b_{s,j},2}(\beta y_j)e^{-\beta y_j}
&=e^{-\beta \sum_{j=1}^l y_j}\prod_{j=1}^l  P_{a_{s,j},b_{s,j},2}(\beta y_j) \\
&=e^{-\beta \sum_{j=1}^l y_{s^{-1}(j)}}\prod_{j=1}^l  P_{a_{1,j},b_{1,j},2}(\beta y_{s^{-1}(j)})\\
&=\prod_{j=1}^l  P_{a_j,b_j,2}(\beta (s^{-1}y)_j)e^{\beta (s^{-1}y)_j}\,.
\end{align*}
Since $F_\phi(sy)=\sgn(s) F_\phi(y)$ and $\h\cap \tau(\Wv)$ is $W(\G,\h)$-invariant, the second equality follows.  
\end{proof}
\begin{cor}
\label{ul, ul' 2}
The distribution $T(\check\Theta_\Pi)$ is non-zero if and only if the highest weight $\lambda=\sum_{j=1}\lambda_je_j \in i\h^*$ of $\Pi$ satisfies the condition
\begin{equation}\label{ul, ul' lambda}
\lambda_1\geq\lambda_2\geq\dots\geq\lambda_l\geq \frac{l'}{2}.
\end{equation}
Equivalently, if and only if the Harish-Chandra parameter $\mu$ of $\Pi$ satisfies
\begin{equation}\label{condition mu l<=l'}
\mu_j\in \delta +\ZZ_{\geq 0} \qquad (1\leq j\leq l)\,
\end{equation} 
where $\ZZ_{\geq 0}$ denotes the set of non-negative integers. 
\end{cor}
\begin{proof} 
Write $\mu=\lambda+\rho$ and $\lambda_j=\frac{l'}{2}+\nu_j$ as in \eqref{lambda} and \eqref{rho}.
Then \eqref{ul, ul' lambda} means that the integers $\nu_j$ satisfy
$\nu_1\geq \nu_2\geq \dots \geq \nu_l\geq 0$. Since
$$
\mu_j=\lambda_j+\rho_j=\frac{l'}{2}+\nu_j+\frac{l+1}{2}-j=\delta+\nu_j+(l-j)
$$
and $\mu$ is strictly dominant, the conditions \eqref{ul, ul' lambda} and \eqref{condition mu l<=l'} are equivalent. 

Notice that, since $b_j=\mu_j-\delta+1$, the condition \eqref{condition mu l<=l'} is also equivalent to 
$b_j\geq 1$ for all $1\leq j\leq l$. 
If the distribution $T(\check\Theta_\Pi)$ is non-zero, then none of the $P_{a_j,b_j,2}$ can be 
identically $0$. So $b_j\geq 1$ for all $1\leq j\leq l$ by \eqref{Pab2}.  

It remains to show that the condition $b_j\geq 1$ for all $1\leq j\leq l$ suffices for the non-vanishing of the right-hand side of \eqref{l<=l' and both compact 1}. 
We are going to use a non-direct argument based on the properties of the Harish-Chandra elliptic orbital integrals \eqref{HC orbital integral on W l<=l'}, though an alternative reasoning will be evident from Proposition \ref{ul, ul' 3} below.

According to \cite[(25)]{McKeePasqualePrzebindaWCestimates}, given $\phi\in \Ss(\Wv)^\G$, there is 
$\psi\in \Ss(\g')$ such that $\phi=\psi\circ \tau'$. 
For $s=(g,g')\in \Sg$ and $w\in \Wv$, we have by \eqref{tau tau'} and \eqref{Ad-action}
$$
\tau'(s.w)=(s.w)^*(s.w)=(gw{g'}^{-1})^*(gw{g'}^{-1})=ig'\overline{w}^tw {g'}^{-1}=\Ad(g')(i\tau'(w))\,.
$$
Hence, for $y=\tau(w)\in \h\cap \tau(\Wv)$,
$$
\int_\Sg \phi(s.w)\, ds=\int_\Sg \psi(\tau'(s.w))\, ds=\frac{\vol(\G)}{\vol(\H)}\, \int_{\G'} \psi\big(\Ad(g')(i\tau'(w))\big)\, dg'\,.
$$
The function
\begin{equation}
\label{part of HC orbital integral on W}
\h\cap \tau(\Wv) \ni y=\tau(w) \to \int_\Sg \phi(s.w)\, ds=\frac{\vol(\G)}{\vol(\H)}\,
\int_{\G'} \psi\big(\Ad(g)(i\tau'(w))\big)\, dg'\in \C  
\end{equation}
is in $\Ss(\h\cap \tau(\Wv))^{W(\G,\h)}$. 
Any function in $\Ss(i\g')^{\G'}$ is of the form 
$$
ix' \to \int_{\G'} \psi\big(\Ad(g')(ix')\big)\, dg' \qquad (x'\in i\h')
$$
for a suitable $\psi:i\g'\to \C$ is in $\Ss(i\g')$
by \cite[Corollary 1.5]{Dadok82} with $\p=i\g'$.
Thus, as $\phi$ ranges in $\Ss(\Wv)^\G$, the function \eqref{part of HC orbital integral on W} 
describes an arbitrary element of $\Ss(\h\cap \tau(\Wv))^{W(\G,\h)}$. 
Therefore
$$
\h\cap\tau(\Wv)\ni y=\tau(w) \to \pi_{\g/\h}(y)\int_\Sg\phi(s.w)\,ds\in \C
$$
is an arbitrary $W(\G,\h)$-skew-invariant element of $\Ss(\h\cap\tau(\Wv))$. Hence, if (\ref{l<=l' and both compact 1}) were zero, then the function 
\[
\prod_{j=1}^l  P_{a_j,b_j,2}(\beta y_j)e^{-\beta y_j}\cdot \frac{\pi_{\g'/\z'}(y)}{\pi_{\g/\h}(y)}
\]
would have to be $W(\G,\h)$-invariant. Equivalently, 
\[
\prod_{j=1}^l  P_{a_j,b_j,2}(\beta y_j)
\]
would have to be $W(\G,\h)$-invariant. This is not possible if $b_j\geq 1$ for all $j$. Indeed,
$\mu$ is strictly dominant and hence, 
by (\ref{Pab2}), the $P_{a_j,b_j,2}$ are non-zero polynomials of different degrees.
Thus, the distribution $T(\check{\Theta}_\Pi)$ is not zero when the condition \eqref{ul, ul' lambda} is satisfied.
\end{proof}
\begin{rem}
The condition \eqref{condition mu l<=l'} means precisely that $\Pi$ occurs in Howe's correspondence, see for example \cite[(A.5.2)]{PrzebindaInfinitesimal}. Recall that we have chosen the Harish-Chandra
parameter $\mu$ to be strictly dominant, i.e. so that $\mu_1>\mu_2>\dots>\mu_l$, but, in fact, the condition \eqref{condition mu l<=l'} does not depend on the choice of the order of roots.
\end{rem}

\begin{pro}
\label{ul, ul' 3}
With the notation of Lemma \ref{ul, ul' 1} and Corollary \ref{ul, ul' 2}, let
\begin{equation}
\label{Pmu}
P_\mu(y)=\prod_{j=1}^{l} P_{a_j,b_j, 2}(\beta y_j) \qquad (y\in \h).
\end{equation}
The distribution $\T(\check{\Theta}_{\Pi})$ is a real analytic $\G\G'$-invariant function on $\Wv$. 
For $w \in \hs1$ it is given by the following formula:
\begin{align}
\label{ul, ul' 3 1}
\T(\check\Theta_{\Pi})(w)&=C_\bullet \,\check{\chi}_\Pi(\t{c}(0)) \,
e^{-\frac{\pi}{2}\langle J(w), w\rangle}\Big(\frac{1}{\pi_{\g/\h}(y)} 
\sum_{s\in W(\G,\h)} \sgn(s)  P_{\mu}(sy)\Big)\\
&=C_\bullet \,\check{\chi}_\Pi(\t{c}(0)) \, 
e^{-\frac{\pi}{2}\langle J(w), w\rangle}\Big(\frac{1}{\pi_{\g/\h}(y)} 
\sum_{s\in W(\G,\h)} \sgn(s)  P_{s\mu}(y)\Big) \nn\,, 
\end{align}
where $C_\bullet$ is a constant depending only on the dual pair, $J=-i1_\Wv$ is the fixed positive compatible complex structure on $\Wv$, $\beta=2\pi$, and $y=\tau(w) \in \h$. The sum in (\ref{ul, ul' 3 1}) is a $W(\G,\h)$-skew symmetric polynomial. Hence its quotient by $\pi_{\g/\h}$ is a $W(\G,\h)$-invariant polynomial on $\h$. This $W(\G,\h)$-invariant polynomial extends uniquely to a $\G$-invariant polynomial $\wt{P}_\mu$ on $\g$. Thus
\begin{equation}\label{ul, ul' 3 1'}
\T(\check \Theta_{\Pi})(w)=C_\bullet \, \check{\chi}_\Pi(\t{c}(0)) \,  e^{-\frac{\pi}{2}\langle J(w), w\rangle}\wt{P}_\mu(\tau(w)) \qquad (w\in \Wv). 
\end{equation}
\end{pro}
\begin{proof}
From Lemma \ref{ul, ul' 1} and formulas \eqref{extension by the symmetry condition1},
\eqref{Ch1} and \eqref{sum yj and J}, we see that for any $\phi\in \Ss(\Wv)$,
\begin{align*}
T(&\check\Theta_\Pi)(\phi)\\
&=C \check{\chi}_\Pi(\t{c}(0)) \int_{\h\cap\tau(\Wv)} e^{-\beta\sum_{j=1}^ly_j}
 P_{\mu}(y) F_\phi(y)\,dy  \\
&=C \check{\chi}_\Pi(\t{c}(0)) \int_{\h\cap\tau(\Wv)} e^{-\beta\sum_{j=1}^ly_j}
\Big(\frac{1}{|W(\G,\h)|}\sum_{s\in W(\G,\h)} \sgn(s) P_{\mu}(sy)\Big) F_\phi(y)\,dy \\ 
&=C C_{\hs1}\check{\chi}_\Pi(\t{c}(0)) \int_{\tau(\h_1^+)} e^{-\beta\sum_{j=1}^ly_j}
\Big(\sum_{s\in W(\G,\h)} \sgn(s) P_{\mu}(sy)\Big) \pi_{\g'/\z'}(y)\int_{\Sg/\Sg^{\hs1}} \phi(s.w)\,ds\,dy \\
&=C_\bullet \check{\chi}_\Pi(\t{c}(0)) \int_{\tau(\h_1^+)}  
e^{-\frac{\pi}{2}\langle J(w), w\rangle}
\wt{P}_{\mu}(y) |\pi_{\so/\hs1^2}(w^2)| \mu_{\Oo(w),\hs1}(\phi)\,d\tau(w)\,,
\end{align*}
where $C_\Wv$ is as in \eqref{CW} and $C_\bullet=C_\Wv^{-1} C C_{\hs1} C(\hs1)$.
The Weyl--Harish-Chandra integration formula on $\Wv$, (\ref{weyl int on w 1}), implies then the result. 
\end{proof}
We now reverse the roles of $\G$ and $\G'$ and compute the intertwining distribution $\T(\check \Theta_{\Pi'})$ for a genuine irreducible unitary representation $\Pi'$ of $\wt \G'$. 
Recall the decomposition $\h'=\h\oplus\h''$ from \eqref{h oplus h''}. Applying the results from 
\cite{McKeePasqualePrzebindaWCSymmetryBreaking} to the dual pair $(\Ug_{l'},\Ug_l)$, we have to replace the form $B$ with $B'$, see \eqref{symmetric form ulul'}. Notice that if $x'=x+x''\in\h'=\h\oplus\h''$ and $y\in \h\subseteq \h'$, then $B'(x',y)=B'(x,y)=-B(x,y)$, where in the last equality we used
the identification of $\h\subseteq \g$ and $\h\subseteq \g'$ coming from \eqref{embedding h into h'}.
This will lead some sign changes in the formulas for $T(\check{\Theta}_{\Pi'})$, as we will detail below. 

Define 
$s_0 \in W(\G',\h')$ by
\begin{equation}
\label{s0-C}
s_0(J'_j)=
\begin{cases}
J'_{l+j} \qquad &(1\leq j\leq l'-l)\\
J'_{j-l'+l} \qquad &(l'-l+1\leq j\leq l').
\end{cases}
\end{equation}
Equivalently, 
\begin{equation}
\label{s0mu'}
(s_0 \mu')_j=\mu'_{s_0^{-1}(j)}=\begin{cases}
\mu'_{l'-l+j} &(1\leq j\leq l)\\
\mu'_{j-l} &(l+1\leq j\leq l')\,.
\end{cases}
\end{equation}
Set
\begin{alignat}{3}
\label{eq:a'jb'j}
&&a_{s_0,j}=-(s_0\mu')_j-\delta+1,\qquad &&b_{s_0,j}=(s_0\mu')_j-\delta+1 \qquad 
&(1\leq j\leq l')\,,\\
\label{eq:sa'jb'j}
&&a'_{s,j}=-(s\mu')_j-\delta'+1,\qquad &&b'_{s,j}=(s\mu')_j-\delta'+1 \qquad 
&(s'\in W(\G',\h'), 1\leq j\leq l')\,.
\end{alignat}
where $\delta$ is as in \eqref{delta-beta} and 
\begin{equation}
\label{delta'}
\delta'=\frac{l-l'+1}{2}\,.
\end{equation}
Notice that $s_0$ is the identity map if $l=l'$. In this case, the symbols $a_{s_0,j}$ and $b_{s_0,j}$ coincide with those introduced in \eqref{eq:ajbj}.  Moreover, $a_{s_0,j}=a'_{s_0,j}$ and $b_{s_0,j}=b'_{s_0,j}$. If $l<l'$, then there is no overlapping notation between \eqref{eq:a'jb'j} and \eqref{eq:ajbj} beacuse $s_0\notin W(\G,\h)$.

We will treat separately the cases $l'>l$ and $l'=l$.

\begin{pro}
\label{TThetaPi' l'>l}
Suppose that $l'>l$ and let $\Pi'$ be a genuine representation of $\wt{\G'}$ with Harish-Chandra parameter $\mu'\in i{\h'}^*$. Then $T(\check \Theta_{\Pi'})\neq 0$ if and only if the following conditions hold:
\begin{align}
\label{condition1 mu' l'>l}
-(s_0\mu')|_{\h}&\in \delta+\ZZ_{\geq 0}, \\
\label{condition2 mu' l'>l}
(s_0\mu')|_{\h''}&=\rho''\,,
\end{align}
where 
\begin{equation}
\label{rho''}
\rho''=\sum_{j=1}^{l'-l} \big(\frac{l'-l+1}{2}-j\big) e_{l+j}=\sum_{j=1}^{l'-l} \big(\delta-j\big) e_{l+j}
\end{equation}
is the $\rho$-function of 
the group $\Ug_{l'-l}$ diagonally embedded in $\Ug_{l'}$ as $\{1_l\} \times \Ug_{l'-l}$. 

Let
\begin{equation}
\label{Ps0mu'}
P_{s_0\mu'}(y)=\prod_{j=1}^{l} P_{b_{s_0,j},a_{s_0,j}, 2}(\beta y_j) \qquad (y\in \h).
\end{equation}
The distribution $\T(\check{\Theta}_{\Pi'})$ is a smooth $\G\G'$-invariant function on $\Wv$. 
For $w \in {\hs1}^{\rm reg}$, it is given by the following formula:
\begin{align}
\label{TPi' explicit}
\T&(\check \Theta_{\Pi'})(w)\\
&=C'_\bullet \, \check{\chi}_{\Pi'}(\t{c}(0)) 
\,\prod_{j=1}^l \frac{(-(s_0\mu')_j+\delta-1)!}{(-(s_0\mu')_j-\delta)!}\,
e^{-\frac{\pi}{2}\langle J(w), w\rangle}\Big(\frac{1}{\pi_{\g/\h}(y)} 
\sum_{s\in W(\G,\h)} \sgn(s)  P_{s_0\mu'}(\beta sy)\Big) \nn\\
&=C'_\bullet \,\check{\chi}_{\Pi'}(\t{c}(0)) \, 
\prod_{j=1}^l \frac{(-(s_0\mu')_j+\delta-1)!}{(-(s_0\mu')_j-\delta)!}\,
e^{-\frac{\pi}{2}\langle J(w), w\rangle}\Big(\frac{1}{\pi_{\g/\h}(y)} 
\sum_{s\in W(\G,\h)} \sgn(s)  P_{ss_0\mu'}(\beta y)\Big) \,, 
\end{align}
where $C'_\bullet$ is a non-zero constant depending only on the dual pair, for all $1\leq j\leq l$
$$
-(s_0\mu')_j+\delta-1=-b_{s_0,j}  \quad \text{and} \quad -(s_0\mu')_j-\delta=a_{s_0,j}-1
$$
are strictly positive, 
$J=-i1_\Wv$ is the fixed positive compatible complex structure on $\Wv$, $\beta=2\pi$, and $y=\tau(w)=\tau'(w)\in \h$. The sum in \eqref{TPi' explicit} is a $W(\G,\h)$-skew symmetric polynomial. Hence its quotient by $\pi_{\g/\h}$ is a $W(\G,\h)$-invariant polynomial on $\h$. 
\end{pro}
\begin{proof}
We start by following the arguments leading to \cite[Theorem 5]{McKeePasqualePrzebindaWCSymmetryBreaking}
with the role of $\Ug_l$ and $\Ug_{l'}$ reversed. By \cite[Corollary 28]{McKeePasqualePrzebindaWCSymmetryBreaking}, $T(\check{\Theta}_{\Pi'})= 0$ unless there is 
$s\in W(\G',\h')$ such that 
\begin{equation}\label{smu-h''}
(s\mu')|_{\h''}=\rho''.
\end{equation}
Moreover, \eqref{smu-h''} holds if and only if $\mu'$ contains a string of length $l'-l$ equal to the coefficients of $\rho''$, i.e. there is $j_0\in\{0,1,\dots,l\}$ such that 
\begin{equation}
\label{smu-C}
\mu'_{j_0+j}=\rho''_{l+j}  \quad \text{and} \quad s(J_{j_0+j})=J_{l+j} \qquad (1\leq j\leq l'-l)\,.
\end{equation}
Furthermore, if \eqref{smu-h''} holds, then for every  
for any $\phi\in \Ss(\Wv)$
\begin{multline}
\label{another intermediate cor l>l'}
T(\check{\Theta}_{\Pi'})(\phi)= 
C \,\check{\chi}_\Pi(\t{c}(0)) \sum_{s\in W(\G',\h'), \, (s\mu')|_{\h''}=\rho''}
\sgn_{\g'/\h'}(s) \\
\times  \int_{\h'}\xi_{-s\mu'}(\widehat{c_-}(x)) \ch^{l-l'-1}(x) 
\int_{\tau(\reg{\hs1})} e^{-iB(x,y)} F'_{\phi}(y)\,dy\,dx.
\end{multline}
where $C$ is a non-zero constant which depends only on the dual  pair $(\Ug_l, \Ug_{l'})$, 
$\widehat{c_-}$ is as in \cite[Corollary 28]{McKeePasqualePrzebindaWCSymmetryBreaking},
$F'$ is Harish-Chandra elliptic orbital integral on $\Wv$ as in \eqref{HC orbital integral on W l>l'}, 
each consecutive integral is absolutely convergent, and we have used the equality $B'(x,y)=-B(x,y)$.

Let $s\in W(\G',\h')$ and $y\in \tau(\hs1)$ be fixed. By \cite[Lemma 29]{McKeePasqualePrzebindaWCSymmetryBreaking} and the change of variable
$x\to -x$ to compensate the opposite sign for $B(x,y)$, we have (in the sense of distributions on $\tau(\reg{\hs1})$),
\begin{equation}
\label{innerintegral-smu-gen}
\int_{\h}\xi_{-s\mu'}(\widehat{c_-}(x)) \ch^{l-l'-1}(x) e^{-iB(x,y)}\,dx
=\Big(\prod_{j=1}^{l} P_{a'_{s,j},b'_{s,j}}(-\beta y_j)\Big) e^{-\beta\sum_{j=1}^{l} |y_j|},
\end{equation}
where $a'_{s,j}$, $b'_{s,j}$ are as in \eqref{eq:ajbj} with $\mu'$ instead of $\mu$, $\beta=2\pi$ as in \eqref{delta-beta}, and $P_{a',b'}$ is defined in \eqref{Pab}.

Let us first study the support of the right-hand side of \eqref{innerintegral-smu-gen} for some special elements in $W(\G',\h')$. Let $j_0\in \{0,1,\dots, l\}$ and consider the permutation
$s_{j_0}$ in $W(\G',\h')$
defined as follows: 
\begin{equation}
\label{s0-j_0-C}
s_{j_0}(J_j)=
\begin{cases}
J_{j} \qquad &(1\leq j\leq j_0)\\
J_{l-j_0+j} \qquad &(j_0+1\leq j\leq j_0+l'-l)\\
J_{j-l'+l} \qquad &(j_0+l'-l+1\leq j\leq l')\,, 
\end{cases}
\end{equation}
i.e.
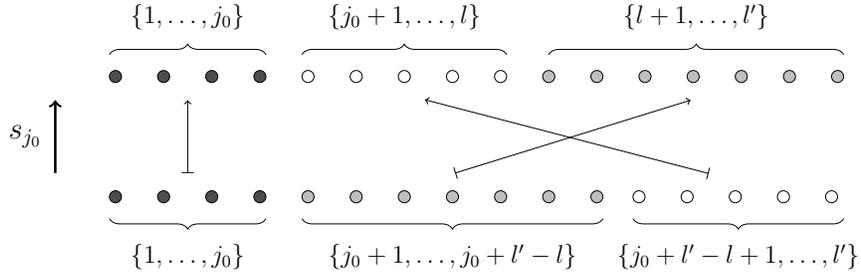
\begin{center}
\scalebox{0.8}{
\tikzset{decorate sep/.style 2 args=
{decorate,decoration={shape backgrounds,shape=circle,shape size=#1,shape sep=#2}}}

\begin{tikzpicture}
\draw[decorate sep={2mm}{8mm},fill=black!70] (0,2) -- (3,2);
\draw[decorate sep={2mm}{8mm},fill=gray!50] (7.2,2) -- (12,2);
\draw[decorate sep={2mm}{8mm}] (3.2,2) -- (6.5,2);

\draw[decorate sep={2mm}{8mm},fill=black!70] (0,0) -- (3,0);
\draw[decorate sep={2mm}{8mm},fill=gray!50] (3.2,0) -- (8.2,0);
\draw[decorate sep={2mm}{8mm}] (8.7,0) -- (12,0);

\draw [decorate,decoration={brace,amplitude=5pt,raise=2ex}]
  (-0.1,2) -- (2.5,2) node[midway,yshift=2.4em]{$\{1,\dots, j_0\}$};
\draw [decorate,decoration={brace,amplitude=5pt,raise=2ex}]
  (3.1,2) -- (6.5,2) node[midway,yshift=2.4em]{$\{j_0+1,\dots, l\}$};
\draw [decorate,decoration={brace,amplitude=5pt,raise=2ex}]
  (7.2,2) -- (12.1,2) node[midway,yshift=2.4em]{$\{l+1,\dots, l'\}$};

\draw [decorate,decoration={brace,amplitude=5pt,mirror,raise=2ex}]
  (-0.1,0) -- (2.5,0) node[midway,yshift=-2.4em]{$\{1,\dots, j_0\}$};
\draw [decorate,decoration={brace,amplitude=5pt,mirror,raise=2ex}]
  (3.1,0) -- (8.1,0) node[midway,yshift=-2.4em]{$\{j_0+1,\dots, j_0+l'-l\}$};
\draw [decorate,decoration={brace,amplitude=5pt,mirror,raise=2ex}]
  (8.6,0) -- (12.1,0) node[midway,yshift=-2.4em]{$\{ j_0+l'-l+1,\dots, l'\}$};

\node at (-1.5,1) (s0) {{\large  $s_{j_0}$}};
\node at (-1,0.25) (s01) {};
\node at (-1,1.75) (s02) {};
\node at (1.2,0.25) (g1) {};
\node at (1.2,1.75) (i1) {};
\node at (5.5,0.35) (g2) {};
\node at (9.7,1.65) (i2) {};
\node at (10,0.35) (g3) {};
\node at (5,1.65) (i3) {};
\draw[->,very thick] 
          (s01) edge (s02) ;
\draw[|->]    
           (g1) edge (i1)
           (g2) edge (i2)
           (g3) edge (i3) ;
\end{tikzpicture}    
}
\end{center}

\noindent Equivalently, 
\begin{equation}
\label{s0mu-Ul}
(s_{j_0} \mu')_j=\mu'_{s_{j_0}^{-1}(j)}=\begin{cases}
\mu'_{j} &(1\leq j\leq j_0)\\
\mu'_{l'-l+j} &(j_0+1\leq j\leq l)\\
\mu'_{j_0-l+j} &(l+1\leq j\leq l')\,.
\end{cases}
\end{equation}
Then $(s_{j_0}\mu')|_{\h''}=\rho''$. 
Moreover, by \cite[Lemma 25]{McKeePasqualePrzebindaWCSymmetryBreaking} with $y_j$ replaced by 
$-y_j$, 
\begin{align}
\label{poly-smu-C}
\prod_{j=1}^{l}& P_{a'_{s_{j_0},j},b'_{s_{j_0},j}}(-\beta y_j) \nn\\
&=\beta^{l}  \Big(\prod_{j=1}^{j_0}  P_{a'_j,b'_j,2}(-\beta y_j) \mathbb{I}_{\R^-}(y_j)\Big) 
\Big(\prod_{j=j_0+1}^l  \!\! P_{a'_{j+l'-l},b'_{j+l'-l},-2}(-\beta y_j) \mathbb{I}_{\R^+}(y_j)\Big) \nn\\
&=\beta^{l}  \Big(\prod_{j=1}^{j_0}  P_{b'_j,a'_j,-2}(\beta y_j) \mathbb{I}_{\R^-}(y_j)\Big) 
\Big(\prod_{j=j_0+1}^l  \!\! P_{b'_{j+l'-l},a'_{j+l'-l},2}(\beta y_j) \mathbb{I}_{\R^+}(y_j)\Big)\,.
\end{align}
has support equal to the closure of $\big(\sum_{j=1}^{j_0} \R^- J_j \big)\oplus \big( \sum_{j=j_0+1}^l \R^+ J_j \big)$. This support is equal to $\tau(\hs1)$ if and only if $j_0=0$, and its intersection with $\tau(\hs1)$ has empty interior otherwise. 
Notice that, by the dominance of $\mu'$ and the definition of $s_0$, for $j_0=0$ we have:
\begin{equation}
\label{mu' s0mu'}
0\geq \delta'=\rho''_{l'}=\mu'_{l'-l}>\mu'_{l'-l+1}=(s_0\mu')_1>\cdots>
\mu'_{l'}=(s_0\mu')_l,
\end{equation}
where the $\mu'_j$'s and $\delta'$ are either all in $\ZZ$ or all in $\ZZ+\frac{1}{2}$. 
Since $-\delta'+1=\delta$, the inequalities \eqref{mu' s0mu'} are equivalent to \eqref{condition1 mu' l'>l}. Moreover, 
$$
0\geq b'_{s_0,1}>\cdots>b'_{s_0,l}.
$$
Since 
\begin{equation}
\label{a'+b'}
a'_{s_0,j}+b'_{s_0,j}=-2\delta'+2=l'-l+1\,,
\end{equation}
we also have
\begin{equation}
\label{a'>=1}
1\leq a'_{s_0,1}<\cdots<a'_{s_0,l}\,.
\end{equation}

An argument as in \cite[Lemma 26]{McKeePasqualePrzebindaWCSymmetryBreaking} shows that if $\mu'$ and $s\in W(G',\h')$ satisfy \eqref{smu-C} for $j_0\in \{1,\dots,l\}$ then the intersection of the support of $\prod_{j=1}^{l}P_{a'_{s,j},b'_{s,j}}$ with $\tau(\hs1)$ has empty interior. 
By \eqref{another intermediate cor l>l'}, this means that $T(\check{\Theta}_{\Pi'})=0$ unless 
\eqref{smu-C} holds with $j_0=0$:
\begin{equation}
\label{smu-C-bis}
\mu'_{j}=\rho''_{l+j} \qquad (1\leq j\leq l'-l), \quad \text{i.e.} \quad
(s_0\mu')_j=\rho''_j \qquad (l+1\leq j\leq l')
\,,
\end{equation}
and in this case, the sum on the right-hand side of \eqref{another intermediate cor l>l'} is over all 
$s\in W(\G',\h')$ satisfying $s(J_j)=J_{l+j}$ for all $1\leq j\leq l'-l$. Moreover, for such an $s$, 
\begin{equation}
\label{ss0inv}
ss_0^{-1}|_{\h''}=1 \qquad \text{and} \qquad ss_0^{-1}(\h)=\h.
\end{equation}
The condition $ss_0^{-1}(\h)=\h$ and the identification \eqref{embedding h into h'}
allow us to consider $ss_0^{-1}$ as permutation of $\{J_1,\dots,J_l\}$ and hence an element of
$W(\G,\h)$. Furthermore, in this case, the contribution to \eqref{another intermediate cor l>l'} from $s$ agrees with that of $s_0$.


The analogue of \cite[(169), (170)]{McKeePasqualePrzebindaWCSymmetryBreaking} is the following formula, which holds for every $\phi\in \Ss(\Wv)$:
\begin{equation}
\label{another intermediate cor l'>l bis}
T(\check{\Theta}_{\Pi'})(\phi)
= C \,\check{\chi}_{\Pi'}(\t{c}(0)) \int_{\tau(\reg{\hs1})} 
\Big(\prod_{j=1}^{l} P_{b'_{s_0,j},a'_{s_0,j},2}(\beta y_j)\Big) 
e^{-\beta\sum_{j=1}^{l} |y_j|}
F'_{\phi}(y)\,dy\,,
\end{equation}
where $C$ is a non-zero constant which depends on the dual pair $(\Ug_l,\Ug_{l'})$ and the product
of polynomials is nonzero because of \eqref{a'>=1} and \eqref{Pab2}.

By \eqref{HC orbital integral on W l<=l'}, \eqref{HC orbital integral on W l>l'} and \eqref{products of roots g'/z'}, we see that, up to a constant of absolute value 1,
$$
F'(y)=\prod_{j=1}^l y_j^{l-l'} F(y) \qquad (y=\tau(w)=\tau'(w)\in \h)\,.
$$
Observe that, for $1\leq j\leq l$,
\begin{equation}
\label{a',b',a,b}
a'_{s_0,j}=a_{s_0,j}+(\delta-\delta') \quad \text{and} \quad 
b'_{s_0,j}=b_{s_0,j}+(\delta-\delta')\,.
\end{equation}
Moreover, for all $1\leq j\leq l$, we have $a_{s_0,j}=-(s_0\mu')_j+\delta'\geq 1$ by \eqref{mu' s0mu'}, 
$b_{s_0,j}=(s_0\mu')_j+\delta'<0$  by \eqref{condition1 mu' l'>l}, 
and $\delta-\delta'=l'-l\geq 0$.
Lemma \ref{lem:shift parameters Pab2} with $a=b_{s_0,j}$, $b=a_{s_0,j}$ and $c=\delta-\delta'$ yields
\begin{equation}
y_j^{l-l'} P_{b'_{s_0,j},a'_{s_0,j},2}(\beta y_j)=2_j^{l-l'} \frac{(-b_{s_0,j})!}{(a_{s_0,j}-1)!} P_{b_{s_0,j},a_{s_0,j},2}(\beta y_j)  \qquad (1\leq j\leq l, \, y\in \h)\,.
\end{equation}
Since $\h\cap \tau(\Wv)$ is the closure of $\tau(\reg{\hs1})$, we obtain, for a possibly different constant $C$, 
\begin{equation}
\label{another intermediate cor l'>l ter}
T(\check{\Theta}_{\Pi'})(\phi)
= C \,\check{\chi}_{\Pi'}(\t{c}(0)) \, 
\prod_{j=1}^{l} \frac{(-b_{s_0,j})!}{(a_{s_0,j}-1)!}
\int_{\h\cap \tau(\Wv)} 
\Big(\prod_{j=1}^{l} P_{b_{s_0,j},a_{s_0,j},2}(\beta y_j)\Big) 
e^{-\beta\sum_{j=1}^{l} |y_j|}
F_{\phi}(y)\,dy\,.
\end{equation}
Formula \eqref{TPi' explicit} is then obtained, by replacing $a_{s_0,j}$ and $b_{s_0,j}$ with their defintions and by computations similar to those in the proof of Proposition \ref{ul, ul' 3}.
The formula also shows that $T(\check{\Theta}_{\Pi'})\neq 0$ whenever $\mu'$ satisfies the conditions 
\eqref{condition1 mu' l'>l} and \eqref{condition2 mu' l'>l}.
\end{proof}

If $l=l'$, then Corollary \ref{ul, ul' 2} and Proposition \ref{ul, ul' 3} apply. However, as in the case $l<l'$, reversing the roles of $\G$ and $\G'$ changes the form $B$ to $-B$. Hence, one has to replace $\mu'$ with $-\mu'$ in the analog of the novanishing condition \eqref{condition mu l<=l'}
for $T(\check{\Theta}_{\Pi'})$. Replacing $\mu'$ with $-\mu'$ also exchanges the indices $a_j$ and $b_j$ in the explicit expression for $T(\check{\Theta}_{\Pi'})$ given in \eqref{ul, ul' 3 1}. We obtain the following proposition. 

\begin{pro}
\label{TThetaPi' l=l'}
Suppose that $l'=l$ and let $\Pi'$ be a genuine representation of $\wt{\G'}$ with Harish-Chandra parameter $\mu'\in i{\h'}^*$. Then $T(\check \Theta_{\Pi'})\neq 0$ if and only 
\begin{equation}
\label{condition1 mu' l'=l}
-\mu'\in \delta'+\ZZ_{\geq 0}.
\end{equation}
Let
\begin{equation}
\label{Pmu' l'=l}
P_{\mu'}(y)=\prod_{j=1}^{l} P_{b'_j,a'_j, 2}(\beta y_j) \qquad (y\in \h).
\end{equation}
The distribution $T(\check{\Theta}_{\Pi'})$ is a smooth $\G\G'$-invariant function on $\Wv$. 
For $w \in {\hs1}^{reg}$, it is given by the following formula:
\begin{align}
\label{TPi' explicit l'=l}
T(\check \Theta_{\Pi'})(w)&=C'_\bullet \, \check{\chi}_{\Pi'}(\t{c}(0)) 
e^{-\frac{\pi}{2}\langle J(w), w\rangle}\Big(\frac{1}{\pi_{\g/\h}(y)} 
\sum_{s\in W(\G,\h)} \sgn(s)  P_{\mu'}(\beta sy)\Big) \nn\\
&=C'_\bullet \,\check{\chi}_{\Pi'}(\t{c}(0)) \, 
e^{-\frac{\pi}{2}\langle J(w), w\rangle}\Big(\frac{1}{\pi_{\g/\h}(y)} 
\sum_{s\in W(\G,\h)} \sgn(s)  P_{s\mu'}(\beta y)\Big) \,, 
\end{align}
where $C'_\bullet$ is a non-zero constant depending only on the dual pair, 
$J=-i1_\Wv$ is the fixed positive compatible complex structure on $\Wv$, $\beta=2\pi$, and $y=\tau(w)=\tau'(w)\in \h$. The sum in \eqref{TPi' explicit l'=l} is a $W(\G,\h)$-skew symmetric polynomial. Hence its quotient by $\pi_{\g/\h}$ is a $W(\G,\h)$ invariant polynomial on $\h$. 
\end{pro}

If $l'=l$, then $s_0$ is the identity map, $\delta'=\delta=1/2$, and $-b_{s_0,j}=a_{s_0,j}-1$ for all 
$1\leq j\leq l$. Hence Proposition \ref{TThetaPi' l'>l} reduces to Proposition \ref{TThetaPi' l=l'} 
once we exclude the empty condition $(s_0\mu')|_{\h''}=\rho''$, see \eqref{condition2 mu' l'>l}. 
With this adjustment, we can and will unify the study of $\T(\check \Theta_{\Pi'})$ in the cases $l'>l$ and $l'=l$.

\begin{cor}
\label{equality intertwining distributions}
Let $\Pi$ and $\Pi'$ be genuine representations of $\G$ and $\G'$ of Harish-Chandra parameters 
$\mu$ and $\mu'$ respectively. Suppose that $\Pi$ satisfies \eqref{condition mu l<=l'}, so that $T(\check{\Theta}_{\Pi})\neq 0$. 
Then $\T(\check \Theta_{\Pi'})$ is a non-zero constant multiple of $\T(\check \Theta_{\Pi})$ if and only if the following conditions are satisfied:
\begin{enumerate}
\thmlist
\item 
there is $s\in W(\G,\h)$ such that 
\begin{equation}
\label{mu mu' relation}
-s\mu=(s_0\mu')|_\h\,, 
\end{equation}
where $s_0\in W(\G',\h)$ is given by \eqref{s0-C};
\item 
$(s_0\mu')|_{\h''}=\rho''$ (when $l'>l$).
\end{enumerate} 
Explicitly, \eqref{mu mu' relation} means that
\begin{equation}
\label{explicit mu mu' relation}
-\mu_j=\mu'_{l'+1-j} \qquad (1\leq j\leq l)\,.
\end{equation}
Regardless of the dominance conventions for $\mu$ and $\mu'$ fixed in \eqref{regular dominant h'} and \eqref{regular dominant h}, $\T(\check \Theta_{\Pi'})$ is a non-zero constant multiple of $\T(\check \Theta_{\Pi})$ if and only if $\mu$ and $\mu'$ can be chosen in their Weyl group orbits so that
\begin{align}
\label{mu mu' Howe1}
\mu'|_{\h}&=-\mu\,,\\ 
\label{mu mu' Howe2}
\mu'|_{\h''}&=\rho'' \qquad (\text{if $l'>l$}).
\end{align}
\end{cor}
\begin{proof}
Condition (b) is part of the nonvanishing of $\T(\check \Theta_{\Pi'})$ when $l'>l$.
By \eqref{ul, ul' 3 1} and \eqref{TPi' explicit}, $\T(\check \Theta_{\Pi'})$ is a non-zero constant multiple of $\T(\check \Theta_{\Pi})$ is and only if there is $s'\in W(\G,\h)$ such that $P_{s'\mu}=P_{s_0\mu}$.
In turn, this holds if and only if there is $s\in W(\G,\h)$ such that $a_{s,j}=b_{s_0,j}$ for all $1\leq j\leq l$, i.e. $-(s\mu)_j=(s_0\mu')_j$ for all $1\leq j\leq l$. This proves \eqref{mu mu' relation}. The chosen dominance orders for $\mu$ and $\mu'$ fix uniquely the element $s\in W(\G,\h)$ for which such an 
equality may hold, namely, $s^{-1}(j)=l-j+1$ for all $1\leq j\leq l$. Together with the definition of $s_0$, this yields \eqref{explicit mu mu' relation}.
\end{proof}

Since, by \eqref{omegaPi},
\[
\OP(\mathcal K( \T(\check\Theta_\Pi)))=\omega(\check\Theta_\Pi)\quad \text{and}\quad 
\OP(\mathcal K (\T(\check\Theta_{\Pi'})))=\omega(\check\Theta_{\Pi'}),
\]
Corollary \ref{equality intertwining distributions} implies the following result, which gives the lists of representations $\Pi\otimes\Pi'$ occurring in Howe's correspondence for $(\Ug_l,\Ug_{l'})$.
\begin{cor}
\label{ul, ul' Howe}
When restricted to the group $\wt\G\wt\G'$, the Weil representation $\omega$ decomposes into the Hilbert direct sum of irreducible components of the form $C_{\Pi\otimes\Pi'}\cdot\Pi\otimes\Pi'$, where $\Pi'$ is determined by $\Pi$ via the conditions (a) and, if $l'>l$, (b) of Corollary \ref{equality intertwining distributions}, and $C_{\Pi\otimes\Pi'}$ is some positive (integral) constant.
\end{cor}

\begin{rem}
\label{addtional constant and dimensions}
Suppose that $\Pi\otimes\Pi'$ occurs in Howe's correspondence. 
Then the misterious factor appearing
in \eqref{TPi' explicit} has the following interpretation:
$$
\prod_{j=1}^l \frac{(-(s_0\mu')_j+\delta-1)!}{(-(s_0\mu')_j-\delta)!}=
(-1)^{l(l-1)/2} \frac{\dim \Pi'}{\dim \Pi} \prod_{j=1}^l \frac{(l'-j)!}{(l-j)!}\,;
$$
see Appendix \ref{appen:dimPi'}.
\end{rem}

\begin{rem} 
\label{rem:central character at tc0}
Let $\Zg=\{\pm 1\}$ be the center of the $\Sp(\Wv)$ and $\wt{\Zg}$ its preimage in $\wt{\Sp}(\Wv)$. 
Since $\Zg\subseteq \G\cap \G'$ and $\t{c}(0)\in \wt{\Zg}$, if $\Pi\otimes\Pi'$ occurs in the restriction of $\omega$ to $\wt\G\wt\G'$, then the central characters of $\Pi$ and $\Pi'$ must coincide:
$\check{\chi}_{\Pi}(\t{c}(0))=\check{\chi}_{\Pi'}(\t{c}(0))$.

Notice that $\t{c}(0)\in \wt{\Sp}(\Wv)$ projects to $-1$ under the metaplectic cover
$\wt{\Sp}(\Wv) \to \Sp(\Wv)$. Hence $\t{c}(0)^2\in\ZZ_2$, the two element kernel of the 
cover. So $\t{c}(0)^4=1$ in $\wt{\Sp}(\Wv)$. It follows that $\check{\chi}_\Pi(\t{c}(0))$ is 
a $4^{\rm th}$-root of unity. 
More precisely, suppose that $\Pi$ has Harish-Chandra parameter $\mu=\lambda+\rho$. 
The center of $\g$ is $\R J_\g=\R\sum_{j=1}^l J_j$, see \eqref{JinG}. Since $\t{c}(0)$ projects to $-1$
under the metaplectic cover, $\t{c}(0)=\wt{\exp}(\pi \sum_{j=1}^l J_j)$, where $\wt{\exp}:\g\to 
\wt{\G}$ is the exponential map. Observe that $\mu(\pi \sum_{j=1}^l J_j)=i\pi \sum_{j=1}^l \mu_j$.
Hence, if $1_\Pi$ denotes the identity map on the space of $\Pi$,
\begin{align*}
\check{\Pi}(\t{c}(0))&=\Pi(\t{c}(0)^{-1})=\Pi\big(\wt{\exp}(-\pi \sum_{j=1}^l J_j)\big)\\
&=\pm e^{\mu(-\pi \sum_{j=1}^l J_j))} 1_{\Pi}
=\pm e^{i\pi \sum_{j=1}^l \mu_j} 1_{\Pi}=\pm (-1)^{\sum_{j=1}^l \mu_j} 1_{\Pi}\,,
\end{align*}
which shows that 
\begin{equation}
\label{central character at wtc0}
\check{\chi}_\Pi(\t{c}(0))=\pm (-1)^{\sum_{j=1}^l \mu_j}\,,
\end{equation}
where the $\pm$ sign is determined by the choice of $\wt{c}$ and hence independent of $\Pi$.
\end{rem}

\section{\bf Multiplicity-one decomposition}

This final section is devoted to the proof that the constant $C_{\Pi\otimes\Pi'}$ occurring in Corollary \ref{ul, ul' Howe} is equal to $1$, i.e. that each irreducible representation 
$\Pi \otimes \Pi'$ of $\wt \G \times \wt \G'$ contained in the oscillatory representation occurs with multiplicity one, see Theorem \ref{constant is one} below. This is a well known and fundamental fact due to Hermann Weyl, \cite{WeylBook}. Our proof is independent of the original one and uses the interwining distributions. Observe that, since the constant $C_{\Pi\otimes\Pi'}$ is a positive integer, it is enough to check that it is one up to a constant of absolute value $1$.

\begin{lem}\label{start constant is one}
Suppose that $\Pi \otimes \Pi'$ occurs in the restriction of $\omega$ to $\wt\G\wt\G'$. Then
$\Pi \otimes \Pi'$ is contained in $\omega$ exactly once 
(equivalently, the constant $C_{\Pi\otimes\Pi'}$ of Corollary \ref{ul, ul' Howe} is equal to 1) 
if and only if
\begin{equation} 
\label{eq:TThetaPi'at0}
\T(\check\Theta_{\Pi})(0)=\vol(\wt\G) \cdot \dim\Pi'\,.
\end{equation}
\end{lem}
\begin{proof}
Let $P_\Pi$ denote the projection of $\omega$ onto its $\wt{\G}$-isotypic component of type $\Pi$. Then 
\begin{align*}
\vol(\wt \G) \cdot  P_\Pi&=
\dim \Pi \cdot \int_{\wt\G} \check\Theta_\Pi(\wt g) \omega(\wt g) \; d\wt g\\
&=\dim \Pi \cdot  \int_{\wt\G} \check\Theta_\Pi(\wt g)  \OP(\mathcal K( T(\wt g)))\; d\wt g\\
&=\dim \Pi \cdot \OP \circ \mathcal K \Big( \int_{\wt\G} \check\Theta_\Pi(\wt g)  T(\wt g)\; d\wt g\Big)\\
&=\dim \Pi \cdot \OP \circ \mathcal K\big(T(\check\Theta_\Pi)\big)\,.
\end{align*}
Also, 
\begin{equation*}
\tr(\OP \circ \mathcal K(T(\check\Theta_\Pi)))=T(\check\Theta_\Pi)(0)
\end{equation*}
by \cite[(148)]{AubertPrzebinda_omega} and \cite[Theorem 3.5.4,(b)] {HoweQuantum}.
It follows that the dimension of the isotypic component of type $\Pi$ is 
\begin{align*} \label{eq:dim-isotypic-Pi}
\tr(P_\Pi)&=\dim \Pi \cdot \tr \big(\OP\Big(\mathcal K(T(\check\Theta_\Pi)\Big)\big)  
\frac{1}{\vol(\wt\G)}\nn\\
&=\dim \Pi \cdot \frac{T(\check\Theta_\Pi)(0)}{\vol(\wt\G)}\,.
\end{align*}
Hence $\Pi \otimes \Pi'$ is contained in $\omega$ exactly once if and only if \eqref{eq:TThetaPi'at0} holds.
\end{proof}

The value $\T(\check\Theta_{\Pi})(0)$ agrees with the value at $0$ of the
the function
$e^{\frac{\pi}{4}\langle J(w), w\rangle}\T(\check\Theta_{\Pi})(w)$, 
determined in Proposition \ref{ul, ul' 3}. The apparent singularity at $0$ due to the division by 
$\pi_{\g/\h}$ will be taken care of by Lemma \ref{lemma:partial pig'h' at 0} below. We first need some 
notation.


Consider the standard inner product $B_s(x,y)=-\tr(xy)$ on $\h$, extended to the complexification 
$\h_\C$ of $\h$ by $\C$-bilinearity. Notice that $B_s(J_j,J_k)=\delta_{j,k}$, the Kronecker delta, 
for all $1\leq j,k \leq l$. For every $\lambda\in \h_\C^*$, let $x_\lambda\in \h_\C$ be the unique element satisfying $\lambda(x)=B_s(x,x_\lambda)$ for all $x\in \h_\C$. For instance, $x_{J^*_j}=J_j$ and 
$x_{e_j}=-iJ_j=E_{j,j}\in i\h$ for every $1\leq j \leq l$. Define $\partial(e_j)$ as the differential operator 
$\partial(E_{j,j})=\partial_{x_j}$ where $x_j=iy_j$ for $y=\sum_{j=1}^l y_j J_j\in \h$. 

Recall that we denote by $\Sigma_l$ the group of permutations of $\{1,2,\dots, l\}$.
The following lemma is a slight modification of a result by Harish-Chandra. We provide a proof for the sake of completeness. 

\begin{lem} \label{lemma:partial pig'h' at 0}
For every smooth function $f:\h \to \C$
\begin{equation} \label{eq:partial pig'h' f at 0}
\big(\partial(\pi_{\g/\h})( \pi_{\g/\h} f)\big)(0)=\Big(\prod_{k=0}^{l} k!\Big) \,f(0)\,.
\end{equation}
\end{lem}
\begin{proof}
In terms of  the coordinates $x_j=iy_j$ ($1\leq j\leq l)$, 
\begin{eqnarray}
\label{eq:partial pig'h'}
\partial(\pi_{\g/\h})&=& \prod_{1\leq j < k \leq l}  
(\partial_{x_j}-\partial_{x_k})
=\sum_{s \in \Sigma_l} \sgn(s) \partial_{x_1}^{s(1)-1} \cdots \partial_{x_l}^{s(l)-1}\,. 
\end{eqnarray}
By the product rule, 
\begin{equation*}
\big(\partial(\pi_{\g/\h})( \pi_{\g/\h} f)\big)(0)=\partial(\pi_{\g/\h})(\pi_{\g/\h}) f(0)\,.
\end{equation*}
Moreover, if $\delta_{s,t}$ denotes Kronecker's delta, then
\begin{align}
\label{eq:partialdelta(delta)}
\partial(\pi_{\g/\h})(\pi_{\g/\h})&=
\sum_{s \in \Sigma_l} \sgn(s) \partial_{x_1}^{s(1)-1} \cdots \partial_{x_{l}}^{s(l)-1}
\Big(\sum_{t \in \Sigma_l} \sgn(t) {x_1}^{t(1)-1} \cdots {x_{l}}^{t(l)-1}  \Big) \nn\\
&=\sum_{s \in \Sigma_l} \sgn(s) \Big(\sum_{t \in \Sigma_l} \sgn(t)\,  \delta_{s,t} \; 
\prod_{k=1}^{l} (t(k)-1)!\Big) \nn\\
&=|\Sigma_l| \; \prod_{k=1}^{l} (k-1)!= \prod_{k=1}^{l} k!\,.
\end{align} 
\end{proof}

\begin{lem}
\label{lem:TThetaPiat0-step1}
Keep the notation of Proposition \ref{ul, ul' 3}. Up to a constant of absolute value (and independent of $\Pi$),  
\begin{equation}
\label{TThetaPiandCbullet}
\Big(\prod_{k=1}^{l} k!\Big) \T(\check\Theta_{\Pi})(0)=
C_\bullet \, \check{\chi}_\Pi(\t{c}(0))^2\, \pi^{\frac{l(l-1)}{2}} 2^{l(l'-1)} |W(\G,\h)| \dim\Pi'\,.
\end{equation} 
\end{lem} 
\begin{proof} 
By Lemma \ref{lemma:partial pig'h' at 0}, $\T(\check\Theta_{\Pi})(0)$ can be computed from 
(\ref{eq:partial pig'h' f at 0}) by evaluating at $0$ the function
$$e^{\frac{\pi}{4}\langle J(w), w\rangle}\T(\check\Theta_{\Pi'})(w)$$
determined in Proposition \ref{ul, ul' 3}. In the computations below, we abbreviate $C_\bullet \,\check{\chi}_\Pi(\t{c}(0))$ as $k_\Pi$.

Set $z_j=2\pi y_j$.  By  \eqref{ul, ul' 3 1} and \eqref{eq:partial pig'h' f at 0},
\begin{align} \label{eq:id at 0 for udud', comp1}
&\Big({\prod_{k=1}^{l} k!\Big)} \T(\check\Theta_{\Pi})(0) \nn\\
&\quad=k_{\Pi}   \partial(\pi_{\g/\h}) \Big( \sum_{t\in W(\G,\h)} \sgn(t) \prod_{j=1}^{l} P_{-(t\mu)_j-\delta+1, (t\mu)_j-\delta+1, 2}(2\pi y_j)\Big)(0)  \nn\\
&\quad=k_{\Pi}   (-2i \pi)^{\frac{l(l-1)}{2}} \prod_{1\leq j < k\leq l} (\partial_{z_j} -\partial_{z_k}) 
\Big( \sum_{t\in W(\G,\h)} \sgn(t) \prod_{j=1}^{l} P_{-(t\mu)_j-\delta+1, (t\mu)_j-\delta+1, 2}(z_j)\Big)(0) \nn\\
&\quad=k_{\Pi}  (-2i \pi)^{\frac{l(l-1)}{2}}  \sum_{t\in W(\G,\h)} \sgn(t)  \sum_{s \in W(\G/\h)} \sgn(s)  
\prod_{j=1}^{l} \big(\partial_{z_j}^{s(j)-1}  P_{-(t\mu)_j-\delta+1, (t\mu)_j-\delta+1, 2}\big)(0) \nn\\
&\quad=k_{\Pi}  (-2i \pi)^{\frac{l(l-1)}{2}}  \sum_{t,s \in W(\G,\h)} \sgn(ts)   
\prod_{j=1}^{l} \big(\partial_{z_j}^{s(j)-1}  P_{-(t\mu)_j-\delta+1, (t\mu)_j-\delta+1, 2}\big)(0)
\nn\\
&\quad=k_{\Pi}  (-2i\pi)^{\frac{l(l-1)}{2}}  |W(\G,\h)| \sum_{s \in W(\G,\h)} \sgn(s)   
\prod_{j=1}^{l} \big(\partial_{z_j}^{s(j)-1}  P_{-\mu_j-\delta+1, \mu_j-\delta+1, 2}\big)(0)
\,.
\end{align}

According to Lemma  \ref{lemma:diff and at  for P}, 
\begin{align}
\partial_{z_j}^{s(j)-1}  P_{-\mu_j-\delta+1, \mu_j-\delta+1, 2}(0)
&= P_{-\mu_j-\delta+1,\mu_j-s(j)-\delta+2, 2}(0) \nn \\
&=(-1)^{\mu_j-\delta-s(j)+1} \; 2^{s(j)+2(\delta-1)} \binom{\mu_j+\delta-1}{s(j)+2(\delta-1)},
\end{align}
where last equality holds under the assumption that $\mu_j\geq \delta-1+s(j)$ for all $1\leq j\leq l$.
Notice that 
\begin{align}
\prod_{j=1}^{l} \binom{\mu_j+\delta-1}{s(j)+2(\delta-1)}
&=\prod_{j=1}^{l} \frac{(\mu_j+\delta-1)!}{(s(j)+2(\delta-1))! (\mu_j-s(j)-\delta+1)!} \nn\\
&=\prod_{j=1}^{l} \frac{1}{(j+2(\delta-1))!} \; 
\prod_{j=1}^{l} \frac{(\mu_j+\delta-1)!}{(\mu_j-s(j)-\delta+1)!} \nn\\
&=\prod_{j=1}^{l} \frac{1}{(l'-j)!} \; \prod_{j=1}^{l} \frac{(\mu_j+\delta-1)!}{(\mu_j-s(j)-\delta+1)!}\,.
\end{align}
Hence, omitting constants of absolute value one and independent of $\Pi$, we obtain
\begin{align} \label{eq:id at 0 for udud', comp2}
\sum_{s \in W(\G,\h)}& \sgn(s)   
\prod_{j=1}^{l} \partial_{z_j}^{s(j)-1}  P_{-\mu_j-\delta+1, \mu_j-\delta+1, 2}(0)  \\
&=2^{ll'-\frac{l(l+1)}{2}} (-1)^{\sum_{j=1}^l \mu_j}\sum_{s \in W(\G,\h)} \sgn(s) \prod_{j=1}^{l}  
\binom{\mu_j+\delta-1}{s(j)+2(\delta-1)} \nn\\
&=2^{ll'-\frac{l(l+1)}{2}} (-1)^{\sum_{j=1}^l \mu_j}\prod_{j=1}^{l} \frac{(\mu_j+\delta-1)!}{(l'-j)!}  
\sum_{s \in W(\G,\h)} \sgn(s) \prod_{j=1}^{l} \frac{1}{(\mu_j-s(j)-\delta+1)!}\,. \nn
\end{align}
Notice that $(-1)^{\sum_{j=1}^l \mu_j}=\pm \check{\chi}_\Pi(\t{c}(0))$ by Remark \ref{rem:central character at tc0}, where the sign is independent of $\mu$.
By (\ref{eq:Vandermonde-1}), 
\begin{equation} \label{eq:id at 0 for udud', comp3}
\sum_{s \in W(\G,\h)} \sgn(s) \prod_{j=1}^{l} \frac{(\mu_j-\delta)!}{(\mu_j-s(j)-\delta+1)!} =
 \prod_{1\leq j < k \leq l} (\mu_j -\mu_k)\,.
\end{equation}
Hence the left-hand side of \eqref{eq:id at 0 for udud', comp2} is equal to
\begin{equation} \label{eq:TThetaPi0-partial1}
2^{ll'-\frac{l(l+1)}{2}} \prod_{j=1}^{l} \frac{1}{(l'-j)!}  \frac{(\mu_j+\delta-1)!}{(\mu_j-\delta)!}  
 \prod_{1\leq j < k \leq l} (\mu_j -\mu_k)\,,
\end{equation}
This is an equality of rational functions, so it holds without the assumptions on the $\mu_j$'s we made to obtain it. Using \eqref{explicit mu mu' relation} and Lemma \ref{dimension of Pi'},
we deduce \eqref{TThetaPiandCbullet}.
\end{proof}

It remains to compute the constant $C_\bullet=C_{\Wv}^{-1}CC_{\hs1}C(\hs1)$ appearing in 
Proposition \ref{ul, ul' 3}, hence the constant $C$ from Lemma \ref{ul, ul' 1} because
both $C_{\hs1}$ and $C(\hs1)$ have absolute value one. In turn, by \eqref{main thm for l<l' a}, 
$C=2(2\pi)^l C_0$ where $C_0$ is the constant computed in \cite[Theorem 2]{McKeePasqualePrzebindaWCSymmetryBreaking}. 
We therefore need to follow all computations that lead us to $C_0$. Some care is needed because the equalities in \cite{McKeePasqualePrzebindaWCSymmetryBreaking} were stated at up non-zero constants, which also allowed us to assume that $\vol(\G)=\vol(\G')=1$.
\label{discussion on constant}

We start by computing the normalization constant $C_\Wv$ appearing in the Weyl--Harish-Chandra integration formula on $\Wv$, see 
\eqref{CW}. 
Recall that the Lebesgue measure on $\Wv=M_{l,l'}(\C)$ is normalized
so that the unit cube with respect to the inner product associated with 
\[
\langle J(w),w\rangle=4\tr(w\overline w^t)
\]
has volume $1$.
Hence our normalization of the Lebesgue measure $\Wv$ is such that 
$dw=4^{ll'} \, dx_{1,1} \, dy_{1,1} \cdots dx_{l,l'} \, dy_{l,l'}$ if 
$w=(w_{j,k})$ with $w_{j,k}=x_{j,k}+iy_{j,k}$ for every $1\leq j\leq l$, $1\leq k\leq l'$.

\begin{lem}
\label{the constant for weyl integration on w moved}
The constant in \eqref{CW} is $C_\Wv=2^{l(l'+1/2)}\,.$

With respect to the fixed Lebesgue measures on $\Wv$ and $\h$, 
the Weyl--Harish-Chandra integration formula on $\Wv$, \eqref{weyl int on w 1} becomes
\begin{multline}
\label{weyl int on w 1, bis}
\int_\Wv \phi(w)\, dw= \frac{C(\hs1) C_\Wv}{\vol(\Sg^{\hs1}) l!}\; 
\int_{(\R^+)^l} \prod_{1\leq j< k\leq l} (y_j-y_k)^2 \prod_{j=1}^l y_j^{l'-l} 
\Big( \int_{\Sg} \phi(s.w) \, ds\Big) dy_1\dots dy_l\\\qquad (\phi\in \mathcal{S}(\Wv))\,.
\end{multline}
\end{lem}
\begin{proof}
We determine the constant $C_\Wv$ by evaluating both sides of \eqref{weyl int on w 1} 
at $\phi(w)=e^{-\tr(w\overline{w}^t)}$.
Then
\[
\int_\Wv \phi(w)\; dw=(4\pi)^{ll'}\,. 
\]
For the right-hand side, observe that, by $\G \times \G'$-invariance,
\begin{equation*}
\int_{\G\times\G'} \phi(s.w)\, ds=\vol(\G)\vol(\G') \phi(w)\,.
\end{equation*}
By \eqref{h1-tautau'}, \eqref{Ch1}, \eqref{products of roots in so}, \eqref{products of roots g/h} and \eqref{products of roots g'/z'}, 
\begin{align}
\label{eq:comp constant superWeyl ulul'}
&\int_{\tau(\hs1^+)} |\pi_{\so/\hs1^2}(w^2)|\int_{\G\times\G'} \phi(s.w) \; ds\,d\tau(w) \nn\\
&=C(\hs1) C_\Wv\vol(\G)\vol(\G') \int_{y_1>\cdots >y_l>0} 
\prod_{1\leq j < k \leq l} (y_j-y_k)^2 \Big(\prod_{j=1}^{l} y_j^{l'-l}\Big) 
e^{-y_1-\dots-y_{l}} \; dy_1\cdots dy_l \nn\\
&=\frac{C(\hs1) C_\Wv \vol(\G)\vol(\G')}{l!} \int_{(\R^+)^{l}} \prod_{1\leq j < k \leq l} (y_j-y_k)^2 \Big(\prod_{j=1}^{l} y_j^{l'-l}\Big) e^{-y_1-\dots-y_{l}} \; dy_1\cdots dy_l\,.
\end{align}
Recall that 
\begin{equation*}
\int_0^\infty y^\alpha e^{-y} \; dy= \alpha!\,.
\end{equation*}
Since
\begin{align*}
\prod_{1\leq j < k \leq l} (y_j - y_k)^2&=
\Big( \sum_{s \in \Sigma_{l}} \sgn(s)\; y_1^{s(1)-1} \cdots y_{l}^{s(l)-1} \Big)^2\\
&= \sum_{s,t \in \Sigma_l} \sgn(st) \; y_1^{s(1)+t(1)-2} \cdots y_{d'}^{s(l)+t(l)-2}\,,
\end{align*}
the integral in \eqref{eq:comp constant superWeyl ulul'} is equal to
\begin{align}
\sum_{s,t \in \Sigma_{l}} \sgn(st) &\int_{(\R^+)^l}  y_1^{s(1)+t(1)+l'-l-2} \cdots y_{l}^{s(l)+t(l)+l'-l-2}  \nn
e^{-y_1-\dots-y_{l}} \; dy_1\cdots dy_{l}  \nn \\
&=\sum_{s,t \in \Sigma_{l}} \sgn(st) \big(s(1)+t(1)+l'-l-2\big)! \cdots \big(s(l)+t(l)+l'-l-2\big)!\nn \\
&= |\Sigma_{l}| \sum_{s\in \Sigma_{l}} \sgn(s) \prod_{j=1}^{l} \big(s(j)+j+l'-l-2\big)! \nn \\
&= l! \det \left[ \big(k+j+l'-l-2\big)! \right]_{j,k=1}^{l}\,.
\end{align}
Applying Lemma \ref{lem:Dan}, we obtain 
\begin{align*}
&\det \left[ \big(k+j+l'-l-2\big)! \right]_{j,k=1}^{l}\\
&\qquad=\Big(\prod_{k=0}^{l-1} (k+l'-l))!\Big) D(3+(l'-l-2),l)=\Big(\prod_{k=l'-l}^{l-1} k! \Big)
\Big(\prod_{k=1}^{l-1} k!\Big)\,,
\end{align*}
The lemma now follows from \eqref{volumes H and G} and \eqref{volume Shs1}.
\end{proof}


\begin{lem}\label{lemma:HC's formula - the constant}
Let $\G=\Ug_l$ and let $B$ be the symmetric $\G$-invariant real bilinear form on $\g$ defined in \eqref{symmetric form ulul'}. Then
\begin{equation*}
\pi_{\g/\h}(x) \pi_{\g/\h}(x')\int_\G e^{iB(g.x,x')}\,dg
=C_\z \sum_{s\in W(\G,\h)}\sgn_{\g/\h}(s)e^{iB(x,sx')}\, \qquad (x,x'\in \h),
\end{equation*}
where $C_\z=i^{-l(l-1)/2} (2\pi)^l$.
\end{lem}
\begin{proof}
Harish-Chandra's formula for the Fourier transform of a regular semisimple orbit, \cite[Theorem 2, page 104]{HC-57DifferentialOperators} computes  
$$
\pi_{\g/\h}(x) \pi_{\g/\h}(x')\int_\G e^{B_\g(g.x,x')}\, dg \qquad (x, x'\in \h)
$$ 
when $B_\g$ is the Killing form of $\g_\C$ and the Haar measure is normalized so that $\vol(\G)=1$. Both sides of that formula are analytic functions of $x,x'$ and hence one can replace $x'$ with $ix'$. Taking into account our normalizations, we find the constant sought for as
$$
C_\z=i^{-l(l-1)/2} (2\pi)^{-l(l-1)/2} \vol(\G) \,\frac{\partial(\pi_{\g/\h})(\pi_{\g/\h})}{|W(\G,\h)|}\,.
$$
\vskip -6mm
\end{proof}

Given $\phi\in \mathcal{S}(\Wv)$, we define $\phi^\G\in\mathcal{S}(\Wv)^\G$ by 
\begin{equation}
\phi^\G(w)=\frac{1}{\vol(\G)} \int_\G \phi(g.w) \, dg\,.
\end{equation}
The following lemmas trace down the multiplicative constants which occur in the formulas from \cite{McKeePasqualePrzebindaWCSymmetryBreaking} according to the normalizations of the present paper. Recall from \cite[Lemma 10]{McKeePasqualePrzebindaWCSymmetryBreaking} that for 
$\mu=\sum_{j=1}^l \mu_j e_j\in i\h^*$,
$$
\xi_{-\mu}(\widehat{c}_-(x))=\prod_{j=1}^l (1+ix_j)^{\mu_j} (1-ix_j)^{-\mu_j} \qquad (x\in \h)\,.
$$

\begin{lem} 
\label{Cor 13 in SBO}
Let $\mu$ be the Harish-Chandra parameter of $\Pi$. Then for any $\phi\in \mathcal{S}(\Wv)$
$$
\int_\G \check{\Theta}_\Pi(\wt{g})T(\wt{g}) \,dg= 
C_1 \int_\h \xi_{-\mu}(\widehat{c}_-(x))\ch^{l'-l-1}(x)
\pi_{\g/\h}(x) \Big(\int_{\Wv} \chi_x(w)\phi^\G(w) \, dw\Big) dx\,,
$$
where 
$$
C_1=i^{ll'} 2^{\frac{l(l+1)}{2}-ll'} \,\frac{\vol(\G)}{\vol(\H)}=i^{ll'} \, \frac{2^{l(l-l')}\pi^{\frac{l(l-1)}{2}}}{\prod_{j=1}^{l-1} j!}\,.
$$
\end{lem}
\begin{proof}
This is \cite[Corollary 13]{McKeePasqualePrzebindaWCSymmetryBreaking}. So we only need to determine $C_1$. Because of Lemma \ref{Cg}, with our normalization of the measures, \cite[Lemma 11]{McKeePasqualePrzebindaWCSymmetryBreaking} reads as follows:
for any $\phi\in \Ss(\Wv)$,
\begin{multline*}
\int_{\G} \check\Theta_\Pi(\t g) T(\t g)(\phi)\,dg\\
=\check\chi_\Pi(\t{c}(0)) \frac{\vol(\G)}{\vol(\H)} \int_\h \xi_{-\mu}(\widehat c_-(x))\cdot \kappa(x)
\cdot\pi_{\g/\h}(x)\Big(\int_\Wv \chi_x(w)\phi^\G(w)\,dw\Big) dx,
\end{multline*}
where
\[
\kappa(x)=\frac{\pi_{\g/\h}(x)}{\Delta(\widehat c_-(x))}\,\Theta(\t c(x))\, j_\g(x)
 \qquad (x\in \h)\,.
\] 
According to \cite[Lemma 17]{McKeePasqualePrzebindaWCSymmetryBreaking}, there is a constant $C$, which depends, only on the dual pair $(\G,\G')$ such that
\[
\kappa(x)=C\ch^{l'-l-1}(x) \qquad (x\in\h)\,.
\]
The constant $C$ can be computed as in the proof of the lemma just quoted. Indeed, $j_\g$ is as in Lemma \ref{Cg}, 
$$
\Theta(\t c(x))=\left(\frac{i}{2}\right)^{\frac{1}{2}\dim \Wv}
\ch^{l'}(x) \qquad (x\in \h)\,.
$$
and, 
by \cite[Lemma 5.7]{PrzebindaUnipotent},
$$
\pi_{\g/\h}(x)=2^{-l(l-1)/2} \Delta(\widehat c_-(x)) \ch^{l-1}(x) \qquad (x\in\h)\,.
$$
Thus $C=i^{ll'} 2^{\frac{l(l+1)}{2}-ll'}$, and the expression for $C_1$ follows. 
\end{proof}

The inner integral on the right-hand side of the equality in Lemma \ref{Cor 13 in SBO} additionally contributes to the constant we are looking for.
Recall the Harish-Chandra regular almost semisimple orbital integral $F(y)$, $y\in\h$, defined in
\eqref{HC orbital integral on W},  and the notation $F_\phi(y)$ for $F(y)(\phi)$.
\begin{lem}\label{reduction to ss-orb-int}
Keep the above notation. Then 
\[
\pi_{\g/\h}(x)\int_\Wv\chi_x(w)\phi^\G(w)\,dw = C_2 \int_{\h\cap \tau(\Wv)}
e^{iB(x,y)} F_{\phi}(y)\,dy\,. 
\]
where 
\begin{equation}
\label{C2}
C_2=C_\z C_{\Wv} \, \frac{i^{-\dim(\g/\h)}}{\vol(\G)}=(-1)^{l(l-1)/2} \, \frac{(2\pi)^l 2^{l(l'+1/2)}}{\vol(\G)}\,.
\end{equation}
\end{lem}
\begin{proof}
This is \cite[Lemma 20]{McKeePasqualePrzebindaWCSymmetryBreaking}, and we only need to 
determine $C_2$. By the Weyl--Harish-Chandra integration formula on $\Wv$, 
see \eqref{weyl int on w 1} and \eqref{CW},
\begin{align}
\label{WIF-1}
\int_\Wv\chi_x(w)\phi^\G(w)\,dw&=C(\hs1) \int_{\tau(\hs1^+)}\pi_{\g/\h}(\tau(w))\pi_{\g'/\z'}(\tau(w))\mu_{\Oo(w),\hs1}(\chi_x\phi^\G)\,d\tau(w) \nn\\
&=C(\hs1)C_\Wv \int_{\tau(\hs1^+)}\pi_{\g/\h}(y)\pi_{\g'/\z'}(y)\mu_{\Oo(w),\hs1}(\chi_x\phi^\G)\,dy
\,.
\end{align}

Recall that we consider $\h$ as embedded in $\h'$ according to \eqref{h oplus h''}. The
centralizer of $\h$ in $\G'$ is $\Zg'=\H\times \Ug_{l'-l}$, whereas
$\Sg^{\hs1}=\Delta(\H)\times\Ug_{l'-l}$. Endow $\Sg/(\H\times\Zg')$ with the quotient measure,
which agrees with the product measure of the quotient measures of $\G/\H$ and $\G'/\Zg'$.
Then, for  every $\phi\in \mathcal{S}(\Wv)$ and $w\in \hs1$, 
\begin{align}
\label{Delta-versus-HH}
\int_{\Sg/\Sg^{\hs1}}& \phi(s.w)\,d(s\Sg^{\hs1})\nn\\
&=\frac{1}{\vol(\Sg^{\hs1})} \int_{\Sg} \phi(s.w)\,ds
\nn\\
&=\frac{1}{\vol(\Sg^{\hs1})} \int_{\Sg/(\H\times\Zg')} \Big( \int_{\H\times \Zg'} \phi\big(s(h,z').w\big)
d(h,z') \Big) d(s(\H\times\Zg')) \nn\\
&=\frac{1}{\vol(\Sg^{\hs1})} \int_{\G/\H\times\G'/\Zg'} 
\Big( \int_{\H\times \H\times \Ug_{l'-l}} \phi\big((gh_1,g'(h_2,u)).w\big) dh_1\, dh_2\, du\Big)
d(g\H)d(g'\Zg')\,.
\end{align}
We apply \eqref{Delta-versus-HH} to $\phi=\chi_x \phi^\G$.

Set $y=\tau(w)$. Then for $s=(g,g')$, where $g\in\G$ and $g'\in\G'$,
\begin{equation}
\label{chix}
\chi_x(s.w)=e^{i\frac{\pi}{2}\langle x(s.w), s.w\rangle} = e^{iB(x, \tau(s.w))}=e^{iB(x, g.\tau(w))}=e^{iB(x, g.y)}
\end{equation}
is independent of $g'$.
If $h_1\in \H$ and $w\in \hs1$, then
$h_1.w=h_1w=w(h_1,1)=(h_1,1)^{-1}.w$ with $(h_1,1)^{-1}\in \G'$. Since $\G$ and $\G'$ commute, 
$$
\chi_x((gh_1,g'(h_2,u)).w)=\chi_x(g.w)=e^{iB(x, g.y)}\,.
$$
Similarly, 
\begin{equation}
\label{phiGs}
\phi^\G(s.w)=\phi^\G(g'.w)
\end{equation}
is independent of $g$. If $(h_2,u)\in \H\times\Zg'$ and $w\in \hs1$, then
$(h_2,u).w=(h_2,1).w=w(h_2,1)^{-1}=h_2^{-1}w=h_2^{-1}.w$, with $h_2\in \H\subseteq \G$.
Hence
$$
\phi^\G((gh_1,g'(h_2,u)).w)=\phi^\G(g'.w)\,.
$$
Thus \eqref{Delta-versus-HH} becomes
\begin{align}
\label{orbital integral chix phiG}
\int_{\Sg/\Sg^{\hs1}} (\chi_x\phi^\G)(s.w)\,d(s\Sg^{\hs1})&=\frac{\vol(\H)\vol(\Zg')}{\vol(\Sg^{\hs1})} \int_{\G/\H}  e^{iB(x, g.y)} d(g\H) \int_{\G'/\Zg'} \phi^\G(g'.w) \, d(g'\Zg') \nn\\
&=2^{-l/2}\int_{\G}  e^{iB(x, g.y)} dg \int_{\G'/\Zg'} \phi^\G(g'.w) \, d(g'\Zg')\,,
\end{align}
because $\frac{{\rm vol}(\H)}{{\rm vol}(\Delta(\H))}=2^{-l/2}$.
In particular, for $x=0$, the equation \eqref{orbital integral chix phiG} gives
\begin{align}
\label{orbital integral phiG}
\int_{\Sg/\Sg^{\hs1}}(\phi^\G)(s.w)\,d(s\Sg^{\hs1})
=2^{-l/2} \vol(\G) \int_{\G'/\Zg'} \phi^\G(g'.w) \, d(g'\Zg')\,.
\end{align}
Since the measure $d(s\Sg^{\hs1})$ is $\G$-invariant, the integral on the left-hand side of 
\eqref{orbital integral phiG} does not change if we replace $\phi^\G$ with $\phi$.
Hence, setting $y=\tau(w)=\tau'(w)$, we obtain
\begin{equation}
\label{HCorbital-integral-with-constants}
C(\h_1) 2^{-l/2} \pi_{\g'/\z'}(y) \int_{\G'/\Zg'} \phi^\G(g'.w) \, d(g'\Zg')
=\frac{i^{-\dim(\g/\h)}}{\vol(\G)} F_{\phi}(y)\,.
\end{equation}
Hence, using \eqref{WIF-1}, \eqref{orbital integral chix phiG}, \eqref{HCorbital-integral-with-constants} and Lemma \ref{HCorbital-integral-with-constants}, we obtain
\begin{align*}
&\pi_{\g/\h}(x)\int_\Wv\chi_x(w)\phi^\G(w)\,dw \\
&=2^{-l/2} C(\hs1)  C_\Wv \int_{\tau(\hs1^+)} \pi_{\g/\h}(y) \pi_{\g/\h}(x) 
\Big(\int_\G e^{iB(x,g.y)}\,dg\Big) \pi_{\g'/\z'}(y) \int_{\G'/\Zg'}\phi^\G(g'.w)\,d(g'\Zg')\,dy\\
&=C_\z C_{\Wv} \, \frac{i^{-\dim(\g/\h)}}{\vol(\G)} \,\sum_{t\in W(\G,\h)}\sgn_{\g/\h}(t)
\int_{\tau(\hs1^+)}e^{iB(x,t.y)} F_\phi(y)\,dy\\ \displaybreak[0]
&=C_\z C_{\Wv} \, \frac{i^{-\dim(\g/\h)}}{\vol(\G)} \, \sum_{t\in W(\G,\h)}\sgn_{\g/\h}(t)\int_{\tau(\hs1^+)} e^{iB(x,y)} F_{\phi}(t.y)\,dy\\
&=C_\z C_{\Wv} \, \frac{i^{-\dim(\g/\h)}}{\vol(\G)} \int_{W(\G,\h)\tau(\hs1^+))}e^{iB(x,y)} F_{\phi}(y)\,dy\\
&= C_2\int_{\h\cap \tau(\Wv)}e^{iB(x,y)} F_{\phi}(y)\,dy\,,
\end{align*}
where $C_2$ is as in \eqref{C2}.
\end{proof}

We can finally prove our main theorem.

\begin{thm}\label{constant is one}
Suppose that $\Pi \otimes \Pi'$ occurs in the restriction of $\omega$ to $\wt\G\wt\G'$; 
see Corollary \ref{ul, ul' Howe}.
Then
\begin{equation} 
\label{eq:TThetaPi'at0}
\T(\check\Theta_{\Pi})(0)=\vol(\wt\G) \cdot \dim\Pi'\,.
\end{equation}
Thus $\Pi \otimes \Pi'$ is contained in $\omega$ exactly once.
\end{thm}
\begin{proof}
By Lemmas \ref{start constant is one} and \ref{lem:TThetaPiat0-step1}, we have to show that, up to a constant of absolute value 1, 
\begin{equation}
\label{constant-to-prove}
C_\bullet \frac{2^{ll'-\frac{l(l+1)}{2}} |W(\G,\h)|}{\prod_{k=1}^{l} k!}=\vol(\wt\G)=2\vol(\G).
\end{equation} 
By the discussion on page \pageref{discussion on constant}, up to a constant of absolute value one, 
$C_\bullet=2(2\pi)^l C_\Wv^{-1} C_1 C_2$, where $C_1$ and $C_2$ are the constants appearing in Lemmas \ref{Cor 13 in SBO} and \ref{reduction to ss-orb-int}, respectively. Then \eqref{constant-to-prove} follows.
\end{proof}

%
\setcounter{section}{0}
\setcounter{equation}{0}
\setcounter{thh}{0}
\appendix
\renewcommand{\thethh}{\Alph{section}.\fontindex{thh}}
\renewcommand{\theequation}{\Alph{section}.\fontindex{equation}}

\section{\bf Some properties of the polynomials $P_{a,b}$}
\label{appen:Pab}

Recall from \eqref{Pab2} that for $a,b\in \ZZ$ with $b\geq 1$ and $\xi\in \R$:
\begin{align}
\label{Pab2others}
P_{a,b,2}(\xi)
&=\sum_{k=0}^{b-1}\frac{a(a+1)\dots (a+k-1)}{k!(b-1-k)!} 2^{-a-k}\xi^{b-1-k} \nn\\
&=\sum_{k=0}^{b-1}(-1)^k\frac{(-a)(-a-1)\dots (-a-k+1)}{k!(b-1-k)!} 2^{-a-k}\xi^{b-1-k} \nn\\
&=\sum_{k=0}^{b-1}(-1)^k\frac{(-a)!}{(-a-k)!k!(b-1-k)!} 2^{-a-k}\xi^{b-1-k}\,.
\end{align}

\begin{lem}
\label{lem:shift parameters Pab2}
Suppose $a,b,c\in\ZZ$ are such that
\begin{equation}
b\geq 1,\  a+b+c=1\ \text{and}\ c\geq 0.
\end{equation}
Then
\begin{equation}
\label{shift parameters Pab2}
P_{a,b,2}(\xi)\xi^c=2^c\frac{(b+c-1)!}{(b-1)!} P_{a+c,b+c,2}(\xi) \qquad (\xi \in \R)\,.
\end{equation}
\end{lem}
\begin{proof}
Since $1-a=b+c\geq 1$, we see from \eqref{Pab2others} that for all $\xi\in \R$
\begin{align*}
\frac{(-a)!}{(b-1)!}   P_{1-b,1-a,2}(\xi)
&=\frac{(-a)!}{(b-1)!}  
\sum_{k=0}^{-a}(-1)^k\frac{(b-1)!}{(b-1-k)!k! (-a-k)!} 2^{b-1-k}\xi^{-a-k}\\
&=\frac{(-a)!}{(b-1)!}  
\sum_{k=0}^{b-1}(-1)^k\frac{(b-1)!}{(b-1-k)!k! (-a-k)!} 2^{b-1-k}\xi^{-a-k}\\
&=\xi^{1-a-b} 2^{a+b-1}  \sum_{k=0}^{b-1}(-1)^k\frac{(-a)!}{(-a-k)!k!(b-1-k)!} 2^{-a-k}\xi^{b-1-k}\\
&=\xi^{1-a-b} 2^{a+b-1} P_{a,b,2}\,,
\end{align*}
where the second equality holds because $-a=b+c-1\geq b-1$ and 
$$
\frac{(b-1)!}{(b-1-k)!}=(b-1)(b-2)\cdots (b-k)=0  \quad \text{if $k\geq b$}\,.
$$
\end{proof}
By combining Lemma \ref{lem:shift parameters Pab2} with \eqref{Pab-2}, we obtain the following corollary.
\begin{cor}\label{propD4}
Suppose $a,b,c\in\ZZ$ are such that
\begin{equation}\label{propD3.a}
a\geq 1,\  a+b+c=1\ \text{and}\ c\geq 0.
\end{equation}
Then
\begin{equation}\label{propD3.b}
P_{a,b,-2}(\xi)(-\xi)^c=\frac{(a+c-1)!}{(a-1)!2^c} P_{a+c,b+c,-2}(\xi).
\end{equation}
\end{cor}

The following lemma is an immediate consequence of the definition of $P_{a,b,2}$.
\begin{lem}  \label{lemma:diff and at  for P}
Suppose $a,b \in \ZZ$.  Then the derivative of $P'_{a,b,2}$ looks as follows:
\begin{equation} \label{eq:diffP}
P'_{a,b,2}(\xi)=P_{a,b-1,2}(\xi)\,.
\end{equation}
If $b \geq 1$ then 
\begin{equation} \label{eq:valuePat0}
P_{a,b,2}(0)=2^{1-a-b} \, \frac{a(a+1)\dots (a+b-2)}{(b-1)!}\,.
\end{equation}
If, moreover $a \leq 0$ and $a+b\leq 1$, then 
\begin{equation} \label{eq:valuePat0-bis}
P_{a,b,2}(0)=(-1)^{b-1} \; 2^{1-a-b} \;\binom{-a}{-a-b+1}\,.
\end{equation}
\end{lem}

\section{\bf The Haar measure on $\G$ via the Cayley transform}
\label{appen:Cayley}

Let $\G\in \{\Og_n, \Ug_n, \Sp_n\}$ be the isometry group of the standard hermitian form
on $\V=\DD^n$, where $\DD=\R, \C$ or $\HH$, respectively, and 
let $\g$ denote the Lie algebra of $\G$. 
Recall that the Cayley transform 
$$
c: \g\ni x \to c(x)=(x+1)(x-1)^{-1}\in \G
$$
is a bijection satisfying $c(c(x))=x$. Moreover, $c(-x)=c(x)^{-1}$. Set 
\begin{equation}
\label{r}
r=\begin{cases}
n-1 &\text{if $\G=\Og_n$}\\
n &\text{if $\G=\Ug_n$}\\
n+1/2 &\text{if $\G=\Sp_n$}
\end{cases}
\end{equation}
and recall from \eqref{ch} the map $\ch$. 
Fix a $\G$-invariant inner product on $\g$ and let $dx$ be the Lebesgue measure on $\g$ normalized so that the volume of any unit cube is one. 

\begin{lem}
\label{lemma:Haar via Cayley}
The formula 
$$
\int_\G f(g) \, d\mu_\G(g) = \int_\g f(c(x)) \ch^{-2r}(x) \, dx \qquad (f\in C_c(\G))
$$
defines a Haar measure $\mu_\G$ on $\G$.
\end{lem}
\begin{proof}
Since $\G$ is unimodular, it suffices to prove that $\mu_\G$ is invariant under right translations
by elements of $\G$. Fix $f\in C_c(\G)$ and $y\in \g$. Then 
$$
\int_\G f(gc(-y)) \, d\mu_\G(g) = \int_\g f(c(z)c(-y)) \ch^{-2r}(z) \, dz\,.
$$
Set $c(x)=c(z)c(-y)$. Then $c(z)=c(x)c(y)$. Hence $z=c\big(c(x)c(y)\big)$ and the above integral is equal to 
$$
\int_\g f(c(x)) \ch^{-2r}\big(c(c(x)c(y))\big) \,j_y(x)\, dx\,,
$$
where $j_y(x)$ is the jacobian of the map 
$$
\g\ni x\to  c\big(c(x)c(y)\big) \in \g\,.
$$
A direct computation (see \cite[(10.2.3)]{HoweOscill}) shows that 
\begin{equation}
\label{c of cc by Howe}
c\big(c(x)c(y)\big)=(y-1)(x+y)^{-1}(x-1)+1\,.
\end{equation}
Hence, for $\Delta \in \g$,
\begin{align*}
c\big(c(x+\Delta)&c(y)\big)-c\big(c(x)c(y)\big)\\
&=(y-1)(x+\Delta+y)^{-1}(x+\Delta-1)-(y-1)(x+y)^{-1}(x-1)\\
&=(y-1)[(x+\Delta+y)^{-1}(x+\Delta-1)-(x+y)^{-1}(x-1)]\,.
\end{align*}
Furthermore, 
\begin{align*}
(x+\Delta+y)^{-1}&(x+\Delta-1)-(x+y)^{-1}(x-1)\\
&=(x+\Delta+y)^{-1}[x+\Delta-1-(x+\Delta+y)(x+y)^{-1}(x-1)]\\
&=(x+\Delta+y)^{-1}\Delta (x+y)^{-1}(y+1)\,.
\end{align*}
The linear part of it is the map
$$
\g\ni \Delta \to (y-1)(x+y)^{-1}\Delta (x+y)^{-1}(y+1)\in \g\,.
$$
The argument in \cite[Appendix A]{McKeePasqualePrzebindaWCSymmetryBreaking} shows that the determinant of this map is equal to 
$$
j_y(x)=\big|\det\big((y-1)(x+y)^{-1}\big)\big|^r\,.
$$
Now, by \eqref{c of cc by Howe},
\begin{align*}
\ch^{-2r}&\big(c(c(x)c(y))\big) \,j_y(x)\\
&=\big|\det\big(1-c(c(x)c(y)\big)\big|^{-r}\big|\det\big((y-1)(x+y)^{-1}\big)\big|^r\\
&=\big|\det\big((y-1)(x+y)^{-1}(x-1)\big)\big|^{-r}\big|\det\big((y-1)(x+y)^{-1}\big)\big|^r\\
&=|\det(x-1)|^{-r}=\ch(x)^{-2r}\,,
\end{align*}
and the result follows.
\end{proof}

\section{\bf The dimension of $\Pi'$}
\label{appen:dimPi'}

For $1\leq j\leq j'$, let $e'_j$ be the basis elements of $i{\h'}^*$ defined in \eqref{ej}, and
consider the standard inner product $\inner{\cdot}{\cdot}$ on $i{\h'}^*$ defined by 
$\inner{e'_j}{e'_k}=\delta_{j,k}$, where $\delta_{j,k}$ is equal to 1 if $j=k$ and $0$ otherwise. 


\begin{lem}
\label{dimension of Pi'}
Let $\Pi'$ be a genuine irreducible representation of $\wt{\G'}=\wt{\Ug}_{l'}$ whose Harish-Chandra parameter $\mu'$ satisfies the conditions \eqref{condition1 mu' l'>l} and \eqref{condition2 mu' l'>l}.
Then 
\begin{equation}
\label{dimPi' dualpair}
\dim\Pi'=\frac{1}{\prod_{j=1}^{l} (l'-j)!}  \; \prod_{j=l'-l+1}^{l'} \frac{(\delta-\mu'_j-1)!}{(-\mu'_j-\delta)!}
 \prod_{l'-l+1\leq j < k \leq l'} (\mu'_j -\mu'_k)\,,
\end{equation} 
where $\delta$ is an in \eqref{delta-beta}.
\end{lem}
\begin{proof}
Recall that, in our conventions, $\mu'$ is regular and dominant provided 
\eqref{regular dominant h'} holds and that the positive roots of $(\g'_\C,\h'_\C)$ are the form 
$\alpha'_{k,j}=e'_j-e'_k$, where $1\leq j< k\leq l'$. Hence Weyl's dimension formula becomes 
\begin{equation}\label{eq:dimPi}
\dim \Pi'= \prod_{1\leq j<k\leq l'} \frac{\inner{\mu'}{\alpha'_{k,j}}}{\inner{\rho'}{\alpha'_{k,j}}}
=\prod_{1\leq j<k\leq l'}  \frac{\mu'_j-\mu'_k}{\rho'_j-\rho'_k}
=\prod_{1\leq j<k\leq l'}  \frac{\mu'_j-\mu'_k}{-j+k}
\,,
\end{equation}
where $\rho'=\sum_{j=1}^{l'} \rho'_j e'_j$ is the $\rho$-function of $\Ug_{l'}$.
Decompose $\{(j,k); 1\leq j<k\leq l'\}$ as $I_1\cup I_2 \cup I_3$, where 
\begin{align*}
I_1&=\{(j,k); 1\leq j<k\leq l'-l\}, \\
I_2&=\{(j,k); l'-l+1\leq j<k\leq l'\}, \\
I_3&=\{(j,k); 1\leq j\leq l'-l, \; l'-l+1\leq j<k\leq l'\}.
\end{align*}
By \eqref{condition2 mu' l'>l}, \eqref{rho''} and \eqref{condition1 mu' l'>l}, 
\begin{align*}
\mu'_j&=\delta-j \qquad (1\leq j\leq l'-l)\\
-\mu'_k&=\delta+n_k, \; \text{where $n_k\in \ZZ_{\geq 0}$} \qquad (l'-l+1\leq k\leq l')\,.
\end{align*}
Hence, if $(j,k)\in I_1$, then $\mu'_j-\mu'_k=-j+k$. On the other hand, if $(j,k)\in I_3$, then 
$\mu'_j-\mu'_k=2\delta-j+n_k$. Hence 
$$
\prod_{(j,k)\in I_3} (\mu'_j-\mu'_k)=\prod_{k=l'-l+1}^{l'} \prod_{j=1}^{2\delta-1} 
(2\delta-j+n_k)=\prod_{k=l'-l+1}^{l'} \frac{(2\delta+n_k-1)!}{n_k!}=
\prod_{k=l'-l+1}^{l'} \frac{(\delta-\mu'_k-1)!}{(-\mu'_k-\delta)!}\,.
$$
Observe now that $\prod_{j=1}^{k-1} (-j+k)=(k-1)!$.
Hence
$$
\prod_{(j,k)\in I_1\cup I_3} (-j+k)= \prod_{k=l'-l+1}^{l'} (k-1)!=\prod_{j=1}^l (l'-j)!\,.
$$
\end{proof}
\section{\bf Some generalizations of Vandermonde's determinant}
\label{appen:Vandermonde}
Recall that we denote by $\Sigma_m$ the group of permutations of $\{1,2,\dots, m\}$.

\begin{lem}\label{leftover 2}
Let $z_j \in\C$ for $1 \leq j \leq m$.
Then, with the convention that empty products are equal to 1,
\begin{equation} \label{eq:Vandermonde-1}
\sum_{s \in \Sigma_m} \sgn(s) \prod_{j=1}^m \prod_{k=1}^{s(j)-1} (z_j-k)= 
\prod_{1\leq j < k \leq m} (z_j-z_k)\,.
\end{equation}
\end{lem}
\begin{proof}
The left-hand side is a Vandermonde determinant. Indeed
\begin{eqnarray*}
&&
\sum_{s \in \Sigma_m} \sgn(s) \prod_{j=1}^m \prod_{k=1}^{s(j)-1} (z_j-k)\\
&&\quad=\det
\begin{bmatrix}
1 & (z_1-1) & (z_1-1)(z_1-2) & \dots & (z_1-1)(z_1-2)\dots(z_1-m+1)\\
1 & (z_2-1) & (z_2-1)(z_2-2) & \dots & (z_2-1)(z_2-2)\dots(z_2-m+1)\\
\vdots & \vdots & \vdots &\ddots &\vdots\\
1 & (z_m-1) & (z_m-1)(z_m-2) & \dots & (z_m-1)(z_m-2)\dots(z_m-m+1)
\end{bmatrix}\\
&&\quad=\det
\begin{bmatrix}
1 & z_1 & z_1^2 & \dots & z_1^{m-1}\\
1 & z_2 & z_2^2 & \dots & z_2^{m-1}\\
\vdots & \vdots & \vdots &\ddots &\vdots\\
1 & z_m & z_m^2 & \dots & z_m^{m-1}
\end{bmatrix}\,.
\end{eqnarray*}
This proves the result.
\end{proof}

\begin{lem} \label{lem:Dan}
Let $n \in \ZZ$ with $n\geq 2$ and let $a \in \C$. Set
\begin{multline*}
D(a,n)={\small
\det \begin{bmatrix}
1 & a & a(a+1) & \dots & a(a+1)\cdots (a+n-2) \\
1 & a+1 & (a+1)(a+2) & \dots & (a+1)(a+2)\cdots (a+n-1) \\
\vdots & \vdots & \vdots & \ddots &\vdots\\
1 & a+n-1 & (a+n-1)(a+n) & \dots & (a+n-1)(a+n)\cdots (a+2n-3) 
\end{bmatrix}}\,.
\end{multline*}
Then 
$D(a,n)=\prod_{k=1}^{n-1} k!\,.$
In particular, $D(a,n)$ is independent of $a$.
\end{lem}
\begin{proof}
For $2\leq j \leq n$ we replace the $j$-th row by the difference between  
the $j$-th and the $(j-1)$-th row. We obtain:
\begin{eqnarray*}
D(a,n)&=&
\det \begin{bmatrix}
1 & a & a(a+1) & \dots & a(a+1)\cdots (a+n-2) \\
0 & 1 & 2(a+1) & \dots & (n-1)(a+1)\cdots (a+n-2) \\
\vdots & \vdots & \vdots & \ddots &\vdots\\
0 & 1 & 2(a+n-1) & \dots & (n-1)(a+n-1)\cdots (a+2n-4) 
\end{bmatrix}\\
&=&
\det \begin{bmatrix}
1 & 2(a+1) & \dots & (n-1)(a+1)\cdots (a+n-2) \\
\vdots & \vdots & \ddots & \vdots \\
1 & 2(a+n-1) & \dots & (n-1)(a+n-1)\cdots (a+2n-4) 
\end{bmatrix}\\
&=&(n-1)!
\det \begin{bmatrix}
1 & (a+1) & \dots & (a+1)\cdots (a+n-2) \\
\vdots & \vdots & \ddots & \vdots \\
1 & (a+n-1) & \dots & (a+n-1)\cdots (a+2n-4) 
\end{bmatrix}\\
&=&(n-1)! D(a+1, n-1)\,.
\end{eqnarray*}
Iterating, we conclude
\begin{equation*}
D(n,a)=(n-1)! D(a+1, n-1)= \dots= (n-1)! \cdots 2!\, D(a+n-2,2)=\prod_{k=1}^{n-1} k!
\end{equation*}
because
\begin{equation*}
D(a+n-2,2)=\det \begin{bmatrix} 1 & a+n-2 \\ 1 & a+n-1 \end{bmatrix}=1\,.
\end{equation*}
\end{proof}

\biblio
\end{document}